\documentclass{nm}
\renewcommand\thisyear {200x}

\renewcommand\thisvolume{xx}

\usepackage{charter}
\usepackage[charter]{mathdesign}
\usepackage[ruled]{algorithm}

\begin{document}


\markboth{K Liu et al.}{Convergence of AMIPDG methods for $\boldsymbol{H}(\boldsymbol{curl})$-elliptic problems}
\title{Convergence of Adaptive Mixed Interior Penalty Discontinuous Galerkin Methods for $\boldsymbol{H}(\boldsymbol{curl})$-Elliptic Problems}


\author[K. Liu, M. Tang, X. Q. Xing, L. Q. Zhong]{K. Liu\affil{1}, M. Tang\affil{2}\comma\corrauth, X. Q. Xing\affil{2} and L. Q. Zhong\affil{2} }
\address{\affilnum{1}\ School of Sciece, East China University of Technology, Nanchang, 330013, China\\
	\affilnum{2}\  School of Mathematical Sciences, South China Normal University, Guangzhou, 510631, China}

\emails{{\tt liukai@ecut.edu.cn} (K. Liu), {\tt mingtang@m.scnu.edu.cn} (M. Tang),{\tt xingxq@scnu.edu.cn}(X. Q. Xing), {\tt zhong@m.scnu.edu.cn} (L. Q. Zhong)}

%
%
%

\begin{abstract}
In this paper, we study  the convergence of adaptive mixed interior penalty  discontinuous Galerkin method  for $\boldsymbol{H}(\boldsymbol{curl})$-elliptic problems. We first get the  mixed model of $\boldsymbol{H}(\boldsymbol{curl})$-elliptic problem by introducing a new intermediate variable. Then we discuss the continuous variational problem and  discrete variational problem, which based on interior penalty discontinuous Galerkin approximation. Next, we construct the corresponding posteriori error indicator, and prove the contraction of the summation of the energy error and the scaled error indicator. At last, we confirm and illustrate the theoretical result through some numerical experiments.
\end{abstract}

\ams{65M15,65N12,65N30}

\keywords{Adaptive mixed interior penalty  discontinuous  Galerkin methods, Convergence, $\boldsymbol{H}(\boldsymbol{curl})$-elliptic problems.}

\maketitle


\section{Introduction}

Let $\Omega\subset{\mathbb{R}}^3$ be Lipschitz bounded polygonal domain with a single connected boundary $ \partial\Omega$. We consider the following $\boldsymbol{H}(\boldsymbol{curl})$-elliptic problem
\begin{eqnarray}\label{Equ:1.1}
\nabla\times\mu\nabla\times \boldsymbol{u}
+
\kappa\boldsymbol{u}=\boldsymbol{f} \quad  &\text{in}& \quad \Omega, \\\label{Equ:1.2}
\boldsymbol{u}\times \boldsymbol{n}=0  \quad &\text{on}& \quad \partial\Omega,
\end{eqnarray}
where $\boldsymbol{n}$  is the unit normal vector of the boundary $ \partial\Omega$, $ \boldsymbol{f}\in \boldsymbol{L}^2(\Omega)$, $\mu$ and $\kappa$ are piecewise constants is consistent with the initial partition $\mathcal{T}_0$ for $\Omega$ and satisfy $\mu_1< \mu <\mu_2$ and  $\kappa_1<\kappa<\kappa_2$, here, $\mu_i$ and $\kappa_i  (i=1, 2)$ are positive constants. By introducing an auxiliary variable $ \boldsymbol{p}=\mu \nabla\times \boldsymbol{u}$, then  we get the  mixed scheme with the boundary value problem  \eqref{Equ:1.1}-\eqref{Equ:1.2}
\begin{eqnarray}\label{Equ:1.3.0}
\boldsymbol{p}=\mu \nabla\times \boldsymbol{u}\quad &\text{in}& \quad  \Omega,
\\ \label{Equ:1.3}
\nabla\times \boldsymbol{p}+  \kappa\boldsymbol{u}=\boldsymbol{f} \quad &\text{in}& \quad  \Omega, \\\label{Equ:1.5}
\boldsymbol{u}\times \boldsymbol{n}=0    \quad &\text{on}& \quad \partial\Omega.
\end{eqnarray}

The mixed finite element method is very convenient for processing high-order equations and equations containing two or more unknown functions, which has attracted widespread attention. For mixed finite element method, there are only few research results for Maxwell problem \cite{JogNan14:887} and Maxwell's eigenvalue problem \cite{JiangLiu14:159,Kik87:509,LiuTob15:458}.

Adaptive finite element method  automatically refines and optimizes meshes according to the singularity of solutions. It is a highly reliable and efficient numerical calculation method. At present, the convergence analysis research of the adaptive mixed finite element method for the elliptic equation is relatively complete. Chen, Holst and Xu \cite{ChenHol09:35} proved the convergence analysis of the adaptive mixed finite element algorithm for elliptic equations. Du and Xie \cite{DuXie15:1327} proved the convergence analysis of the adaptive mixed finite element algorithm for the convection diffusion equation. However,  there are only few research results on the posterior error estimator of Maxwell's equations for the adaptive mixed finite element method. For example, Carstensen and Ma \cite{CarstensenMa21} establishes the convergence of adaptive mixed finite element methods for second-order linear non-self-adjoint indefinite elliptic problems. Carstensen, Hoppe, Sharma and  Warburton \cite{CarHop11:13} designs and analyzes the posterior error estimation of the adaptive hybrid conforming finite element method of $\boldsymbol{H}(\boldsymbol{curl})$-elliptic problem. Recently, Chung, Yuen and  Zhong \cite{ChungYuen14:613} present a-posteriori error analysis for the staggered discontinuous Galerkin  method. As far as we know, there are not any published literatures for the convergence analysis of the adaptive mixed finite element method for the boundary value problem\eqref{Equ:1.3.0}-\eqref{Equ:1.5}. Our contributions in this paper are to
\begin{itemize}
	\item construct a new error estimator, which does not include the negative power of the local mesh size in the jump term for the traditional DG method;
	\item  get the convergence of the Adaptive Mixed Interior Penalty  Discontinuous Galerkin  (AMIPDG)  method by using the similar technique used in \cite{BonitoNochetto10:734}. However, this technique in \cite{BonitoNochetto10:734} can not be used directly for mixed forms.
\end{itemize}


We present our main result in the following theorem.
\begin{theorem}
	Let $\{  \mathcal{T}_k, \boldsymbol{U}_k, \boldsymbol{Q}_k, \boldsymbol{u}_{k}, \boldsymbol{p}_k, \eta( \boldsymbol{u}_{k}, \boldsymbol{p}_{k}; \mathcal{T}_{k})\}_{k\geq 0}$ be the sequence of meshes, finite element space, mixed discrete solution and	posterior error estimate indicator produced by the \textbf{AMIPDG} algorithm.  Then there exist constants $\rho>0$ and  $\delta \in(0, 1)$, which depend on marking parameter and the shape regularity of the initial mesh $\mathcal{T}_{0}$, such that
	\begin{eqnarray}\nonumber
	\|\vert\boldsymbol{u}-\boldsymbol{u}_{k+1} \vert\|_{k+1}^2  + \rho   \eta^2( \boldsymbol{u}_{k+1}, \boldsymbol{p}_{k+1}; \mathcal{T}_{k+1} ) \leq  \delta\bigg(\|\vert  \boldsymbol{u}-\boldsymbol{u}_k\vert\|_k^2+ \rho \eta^2( \boldsymbol{u}_k, \boldsymbol{p}_k; \mathcal{T}_k )\bigg).
	\end{eqnarray}
	Therefore, for a given precision, the \textbf{AMIPDG}  method will terminate after a finite number of operations.
\end{theorem}

For convenience, we let $C$ denote a generic positive constant which may be different at different occurrences and adopt the following notation. The subscripted constant $C_i$ represents a particularly important constant. $a\lesssim b$ means $a\leq Cb$ for some constants $C$ which are independent of mesh sizes.

The rest of this paper is organized as follows. In Section 2, we first present the continuous variational problem, the discrete variational problem, and the procedure of AMIPDG. In Section 3, we first show the  upper bound estimate of the error, which is key to the convergence analysis, then we prove the indicator reduction and the convergence of AMIPDG algorithm. In Section 4, we provide some numerical experiments to illustrate the effectiveness of the AMIPDG.

\section{Adaptive Mixed interior penalty  discontinuous Galerkin method}\label{Sec:2}

In this section, we introduce the continuous variational problem, the discrete variational problem of mixed internal penalty discontinuous finite element method, and the procedure of AMIPDG.

\subsection{Continuous variational problem}

For an open and connected bounded domain $D \subset \mathbb{R}^3$, we denote by ${L}^2(D)$ (resp. $\boldsymbol{L}^2(D):=(L^2(D))^{3}$) the spaces of square-integrable functions (resp. vector fields) on $D$ with inner product $(\cdot,\cdot)_{0,D}$.
We define the spaces
\begin{eqnarray*}
	&& \boldsymbol{H} (\boldsymbol{curl};D)=\{ \boldsymbol{u}\in \boldsymbol{L}^2(D):  \nabla\times\boldsymbol{u} \in \boldsymbol{L}^2(D) \},
	\\
	&&\boldsymbol{H} (div;D)=\{ \boldsymbol{u}\in \boldsymbol{L}^2(D):  \nabla\cdot \boldsymbol{u} \in L^2(D) \},
\end{eqnarray*}
with
\begin{eqnarray*}
	&& (\boldsymbol{u}, \boldsymbol{v})_{\boldsymbol{curl}, D}:=(\boldsymbol{u}, \boldsymbol{v})_{0, D}+(\nabla\times\boldsymbol{u}, \nabla\times \boldsymbol{v})_{0, D},
	\quad \forall \boldsymbol{u}, \boldsymbol{v} \in \boldsymbol{H} (\boldsymbol{curl};D),
	\\
	&&(\boldsymbol{u}, \boldsymbol{v})_{div, D}:=(\boldsymbol{u}, \boldsymbol{v})_{0, D}+(\nabla\cdot\boldsymbol{u}, \nabla\cdot\boldsymbol{v})_{0, D},   \quad \forall\boldsymbol{u}, \boldsymbol{v} \in \boldsymbol{H} (div;D),
\end{eqnarray*}
and the induced norm as:
\begin{eqnarray*}
	&&\| \boldsymbol{u}\|^2_{\boldsymbol{curl}, D}:=\|\boldsymbol{u}\|^2_{0, D}+\|\nabla\times \boldsymbol{u}\|^2_{0, D}, \ \forall \boldsymbol{u}\in\boldsymbol{H}(\boldsymbol{curl},D), \\
	&&\|\boldsymbol{u}\|^2_{div, D}:=\| \boldsymbol{u}\|^2_{0, D}+\|\nabla\cdot\boldsymbol{u}\|^2_{0, D},  \  \quad \forall \boldsymbol{u}\in\boldsymbol{H}(div,D),
\end{eqnarray*}
respectively, where $\|\cdot\|_{L^2(D)}:=(\cdot,\cdot)^{1/2}_D$ denotes the norm of the space $L^2(D)$ or $\boldsymbol{L}^2(D)$. We also define $	\boldsymbol{H}_0(\boldsymbol{curl};D)=\{ \boldsymbol{v}\in \boldsymbol{H}(\boldsymbol{curl};D): \boldsymbol{v}\times\boldsymbol{n}=0 \ on \ \partial D  \}$ in the trace sense.

Next,  we first define two space $\boldsymbol{U}:= \boldsymbol{H}_0(curl;\Omega),  \boldsymbol{Q}:=\boldsymbol{L}^2(\Omega)$. Then, the mixed variational problem of the mixed boundary value problem \eqref{Equ:1.3.0}-\eqref{Equ:1.5} reads as: find $ (\boldsymbol{u}, \boldsymbol{p})\in \boldsymbol{U}\times  \boldsymbol{Q}$ such that:
\begin{eqnarray}\label{Equ:3.6}
&& a(\boldsymbol{p}, \boldsymbol{q})-b(\boldsymbol{u}, \boldsymbol{q})=\ell_1(\boldsymbol{q}), \qquad \forall  \boldsymbol{q}\in \boldsymbol{Q}, \\ \label{Equ:3.7}
&&  d(\boldsymbol{v}, \boldsymbol{p})+c(\boldsymbol{u}, \boldsymbol{v})=\ell_2(\boldsymbol{v}), \qquad \forall \boldsymbol{v}\in\boldsymbol{U}.
\end{eqnarray}
The bilinear forms $a, b, c$ and the functionals $\ell_1(\cdot) , \ell_2(\cdot)$ are given by
\begin{eqnarray}\label{Equ:3.8}
&& a(\boldsymbol{p}, \boldsymbol{q}):=(\boldsymbol{p}, \boldsymbol{q}), \\\label{Equ:3.9}
&& b(\boldsymbol{u}, \boldsymbol{q}):=(\mu\nabla\times \boldsymbol{u}, \boldsymbol{q}), \\ \label{Equ:3.10}
&& c(\boldsymbol{u}, \boldsymbol{v}):=( \kappa\boldsymbol{u}, \boldsymbol{v}),
\\ \label{Eqn:3.10}
&& d(\boldsymbol{v},\boldsymbol{p}):=(\nabla\times \boldsymbol{v},\boldsymbol{p})
\\  \label{Equ:3.11}
&& \ell_1(\boldsymbol{q}):=0, \\\label{Equ:3.12}
&& \ell_2(\boldsymbol{v}):=(\boldsymbol{f}, \boldsymbol{v}).
\end{eqnarray}

The operator-theoretic framework involves operator $ \mathcal{A}:(\boldsymbol{U}\times \boldsymbol{Q})\rightarrow (\boldsymbol{U}\times \boldsymbol{Q})^*$  defined by
\begin{equation}\label{Equ:3.13}
( \mathcal{A}(\boldsymbol{u}, \boldsymbol{p}) )(\boldsymbol{v}, \boldsymbol{q}) := a(\boldsymbol{p}, \boldsymbol{q})-b(\boldsymbol{u}, \boldsymbol{q})+d(\boldsymbol{v}, \boldsymbol{p})+c(\boldsymbol{u}, \boldsymbol{v}),  \forall \boldsymbol{u}, \boldsymbol{v}\in \boldsymbol{U} ,   \boldsymbol{p, q} \in \boldsymbol{Q},
\end{equation}
where $(\boldsymbol{Q}  \times \boldsymbol{U})^*$ is the dual spaces of  $(\boldsymbol{Q} \times \boldsymbol{U})$.
Then we can rewrite \eqref{Equ:3.6}-\eqref{Equ:3.7} as
\begin{equation}\label{Equ:3.14}
( \mathcal{A}(\boldsymbol{u}, \boldsymbol{p}) )(\boldsymbol{v}, \boldsymbol{q})=\ell( \boldsymbol{v}, \boldsymbol{q}),
\end{equation}
with $ \ell(\boldsymbol{v}, \boldsymbol{q})=\ell_1(\boldsymbol{q})+\ell_2(\boldsymbol{v})$, and $\ell_i$ are given by \eqref{Equ:3.11}-\eqref{Equ:3.12}.

Then, we state the well-posedness of the variational problem \eqref{Equ:3.6}-\eqref{Equ:3.7} in  the following lemma, and it can be found in section 3 of \cite{CarHop09:27}.
\begin{lemma}\label{Lem:3.0}	
	Under the assumptions on the problem of  \eqref{Equ:1.1}-\eqref{Equ:1.2}, $\mathcal{A}$ is a continuous and bijective linear operator.
	Hence, for any $\ell = (\ell_1, \ell_2)\in (\boldsymbol{Q} \times \boldsymbol{U})^*$, the  mixed variational  problem \eqref{Equ:3.6}-\eqref{Equ:3.7} has a unique solution $(\boldsymbol{u}, \boldsymbol{p}) \in (\boldsymbol{U}\times \boldsymbol{Q})$, which satisfy the following continuously
	\begin{equation}\label{Equ:3.15}
	\|(\boldsymbol{u}, \boldsymbol{p})\|_{\boldsymbol{U}\times \boldsymbol{Q}}
	:=
	(\|\boldsymbol{u}\|^2_{curl,\Omega}+\|\boldsymbol{p}\|^2_{0})^{1/2} \lesssim \|\ell_1 \|_{\boldsymbol{Q}^*}+\|\ell_2\|_{U^*}.
	\end{equation}
\end{lemma}

\subsection{Discrete variational problem}
We suppose that $\mathcal{T}_h$ is a family of shape regularity, quasi-uniform and conform tetrahedral generation on  $ \Omega$. Let $h_ {\tau}=\vert\tau\vert^{1/3}$ denote the mesh size with $\vert \tau \vert$ being the volume of $\tau \in \mathcal{T}_h$.

Define the discontinuous finite element function space $\mathbb{V}(\mathcal{T}_h)$ as:
\begin{eqnarray*}
	\mathbb{V}(\mathcal{T}_h)
	=
	\{ \boldsymbol{v}\in \boldsymbol{L}^2(\Omega):\boldsymbol{v}_\tau =\boldsymbol{v}\vert_\tau \in (P_l(\tau))^3, \quad \forall \tau\in \mathcal{T}_h \},
\end{eqnarray*}
where $P_l(\tau)$ is the set of polynomials defined in the volume $\tau$ whose degree does not exceed $l$, where $l\geq1$ is an integer.

Let $\mathcal{F}_h$, $\mathcal{F}_h^0$ and $\mathcal{F}_h^\partial$ denote the set of the all faces  of its volumes, and the set of internal faces, and the set of boundary faces, respectively. Thus, $\mathcal{F}_h=\mathcal{F}_h^0\bigcup\mathcal{F}_h^\partial$.
Let ${H}^1(\Omega; \mathcal{T}_{h})$ be the space of piecewise Sobolev functions defined by
\begin{equation*}
{H}^1(\Omega; \mathcal{T}_{h})=\left\{\boldsymbol{v}\in{L}^2(\Omega) :  \boldsymbol{v}_\tau=\boldsymbol{v}\vert_\tau\in{H}^1(\tau), \quad \forall\ \tau\in \mathcal{T}_{h}\right\}.
\end{equation*}
and  $ \boldsymbol{H}^1(\Omega; \mathcal{T}_{h}) = ( {H}^1(\Omega; \mathcal{T}_{h}) )^3$. Let $\boldsymbol{L}^2(\mathcal{F}_h)$ be the set of $\boldsymbol{L}^2$ functions defined on $\mathcal{F}_h$. Moreover, we define the following inner products
\begin{eqnarray*}
	(\boldsymbol{v},\boldsymbol{w})_{\mathcal{T}^{'}_h}
	&=&
	\sum\limits_{\tau\in\mathcal{T}^{'}_h}\int_{\tau}\boldsymbol{v}\cdot\boldsymbol{w} \mathrm{d}\boldsymbol{x}, \quad \forall \boldsymbol{v},\boldsymbol{w}\in \boldsymbol{L}^2(\Omega),\ \forall \mathcal{T}^{'}_h \subset \mathcal{T}_h, \\
	<\boldsymbol{v},\boldsymbol{w} >_{\mathcal{F}^{'}_h}
	&=&
	\sum\limits_{f\in\mathcal{F}^{'}_h}\int_{f}\boldsymbol{v}\cdot\boldsymbol{w} \mathrm{d}\boldsymbol{s}, \quad \forall \boldsymbol{v},\boldsymbol{w}\in \boldsymbol{L}^2(\mathcal{F}_h),\ \forall \mathcal{F}^{'}_h \subset \mathcal{F}_h.
\end{eqnarray*}

For $f\in\mathcal{F}^0_h$, we have $\tau_i\in\mathcal{T}_h (i=1, 2)$, such that $f =\partial\tau_1\cap\partial\tau_2$. Then we denote the jump and average of $ \boldsymbol{v}$ as:
\begin{eqnarray*}
	[[ \boldsymbol{v} ]]  &=& \boldsymbol{v}_1\times \boldsymbol{n}_1+ \boldsymbol{v}_2\times \boldsymbol{n}_2, \quad \forall \boldsymbol{v} \in \boldsymbol{H}^1 (\Omega; \mathcal{T}_{h}),
	\\
	\{\{ \boldsymbol{v}  \}\}   &=&\dfrac{\boldsymbol{v}_1 +\boldsymbol{v}_2}{2} , \quad \forall \boldsymbol{v} \in \boldsymbol{H}^1 (\Omega; \mathcal{T}_{h}),
\end{eqnarray*}
where $\boldsymbol{v}_i$ denote the values of $\boldsymbol{v}$ on $\boldsymbol{v}\vert_{\tau_i}(i=1, 2)$ and $\boldsymbol{n}_i$ denote the out unit normal vectors on $f$ exterior $\boldsymbol{v}\vert_{\tau_i}$.

For $f\in\mathcal{F}^{\partial}_h$, we have $\tau\in\mathcal{T}_h$, such that $f = \partial\tau\cap\partial\Omega$. Then we denote the jump and average of $ \boldsymbol{v}$ as:
\begin{equation}
[[ \boldsymbol{v} ]]=\boldsymbol{v}_\tau \times \boldsymbol{n}_{\partial\Omega},\ \ \{\{ \boldsymbol{v}  \}\}=\boldsymbol{v}_\tau.
\end{equation}

Next, we give the corresponding discrete scheme of \eqref{Equ:3.6}-\eqref{Equ:3.7}. Firstly, we define the corresponding discrete space as follow
\begin{eqnarray*}
	&&\boldsymbol{U}_h:= \{ \boldsymbol{v}_h \in \mathbb{V}(\mathcal{T}_h)\vert \quad [[ \boldsymbol{v}_h]]\vert_{f}=0, \forall  f \in \mathcal{F}^\partial_h\}, \\
	&&\boldsymbol{Q}_h:=  \mathbb{V}(\mathcal{T}_h).
\end{eqnarray*}
Then, the formulation of the discrete Mixed Interior Penalty  Discontinuous Galerkin (MIPDG) method reads: find $(\boldsymbol{u}_h, \boldsymbol{p}_h) \in( \boldsymbol{U}_h, \boldsymbol{Q}_h)$ such that
\begin{eqnarray}\label{Equ:3.16}
&& a_h(\boldsymbol{p}_h, \boldsymbol{q}_h)-b_h(\boldsymbol{u}_h, \boldsymbol{q}_h)=\ell_{1, h}(\boldsymbol{q}_h)+d_{1,h}(\boldsymbol{u}_h, \boldsymbol{q}_h), \qquad \forall  \boldsymbol{q}_h\in \boldsymbol{Q}_h,
\\ \label{Equ:3.17}
&&  d_h(\boldsymbol{v}_h, \boldsymbol{p}_h)+c_h(\boldsymbol{u}_h, \boldsymbol{v}_h)=\ell_{2, h}(\boldsymbol{v}_h)+d_{2,h}(\boldsymbol{u}_h, \boldsymbol{v}_h), \qquad \forall \boldsymbol
{v}_h\in \boldsymbol{U}_h,
\end{eqnarray}
where
\begin{eqnarray*}
	&& a_h(\boldsymbol{p}_h, \boldsymbol{q}_h):=(\boldsymbol{p}_h, \boldsymbol{q}_h)_{\mathcal{T}_h},
	\\
	&& b_h(\boldsymbol{u}_h, \boldsymbol{q}_h):=(\mu \nabla \times \boldsymbol{u}_h, \boldsymbol{q}_h )_{\mathcal{T}_h},
	\\
	&& c_h(\boldsymbol{u}_h, \boldsymbol{v}_h):= (\kappa\boldsymbol{u}_h, \boldsymbol{v}_h)_{\mathcal{T}_h},
	\\
	&& d_h(\boldsymbol{v}_h, \boldsymbol{p}_h):= (\nabla\times\boldsymbol{v}_h,\boldsymbol{p}_h)_{\mathcal{T}_h},
	\\
	&& \ell_{1, h}(\boldsymbol{q}_h):=0,
	\\
	&& \ell_{2, h}(\boldsymbol{v}_h):=(\boldsymbol{f}, \boldsymbol{v}_h)_{\mathcal{T}_h},
	\\
	&& d_{1,h}(\boldsymbol{u}_h, \boldsymbol{q}_h):=-< \{\{\mu \boldsymbol{q}_h  \}\} , [[ \boldsymbol{u}_h ]]  >_{\mathcal{F}_h},
	\\
	&& d_{2,h}(\boldsymbol{u}_h, \boldsymbol{v}_h):=< ( \{\{ \mu \nabla \times \boldsymbol{u}_h  \}\} -\alpha h^{-1}_f[[ \boldsymbol{u}_h ]]), [[ \boldsymbol{v}_h ]]  >_{\mathcal{F}_h},
\end{eqnarray*}
here the constant $\alpha>0$ denote the penalty parameter, $h_f$ denote the diameter of the circumcircle of  $f$.
Thus $h_\tau \approx h_f$.
\begin{remark}\label{Rem:c}
	The calculation of $\nabla\times\boldsymbol{u}_h$ in the bilinear terms are  piecewise derivations.
\end{remark}

The standard symmetric Interior Penalty Discontinuous Galerkin  (IPDG) method of the  boundary value problem \eqref{Equ:1.1}-\eqref{Equ:1.2} is to find $ \boldsymbol{u}_h \in  \boldsymbol{U}_h $, such that
\begin{eqnarray} \nonumber
\lefteqn{
	a_{IP}( \boldsymbol{u}_h, \boldsymbol{v}_h)}\\ \nonumber
&&:=
(\kappa\boldsymbol{u}_h, \boldsymbol{v}_h)_{\mathcal{T}_h}
+
(\mu\nabla\times \boldsymbol{u}_h, \nabla\times \boldsymbol{v}_h)_{\mathcal{T}_h}
-
<\{\{\mu \nabla\times \boldsymbol{v}_h \}\} , [[ \boldsymbol{u}_h ]]  >_{\mathcal{F}_h} \\ \label{Equ:3.20}
&&
\quad -
< \{\{\mu \nabla\times \boldsymbol{u}_h   \}\} , [[  \boldsymbol{v}_h]]  >_{\mathcal{F}_h}
+
\alpha h^{-1}_f<[[ \boldsymbol{u}_h ]], [[ \boldsymbol{v}_h ]]  >_{\mathcal{F}_h}\\ \nonumber
&&=(\boldsymbol{f}, \boldsymbol{v}_h)_{\mathcal{T}_h}.
\end{eqnarray}
The following lemma shows that the discrete variational problems  \eqref{Equ:3.16}-\eqref{Equ:3.17} and \eqref{Equ:3.20} are equivalent.
\begin{lemma}\label{Lem:0}[\cite{CarHop09:27}, Theorem 4.1]
	The formulations
	\eqref{Equ:3.16}-\eqref{Equ:3.17} and  \eqref{Equ:3.20} are formally equivalent in the following sense. If $ (\boldsymbol{u}_h, \boldsymbol{p}_h) \in (\boldsymbol{U}_h, \boldsymbol{Q}_h) $ are the solution of discrete variational problem  \eqref{Equ:3.16}-\eqref{Equ:3.17}, then $\boldsymbol{u}_h \in \boldsymbol{U}_h $ solves  \eqref{Equ:3.20}. Conversely, if  $\boldsymbol{u}_h \in \boldsymbol{U}_h $ solves  \eqref{Equ:3.20}, then there exists some $\boldsymbol{p}_h \in  \boldsymbol{Q}_h$ such that $ (\boldsymbol{u}_h, \boldsymbol{p}_h) \in (\boldsymbol{U}_h, \boldsymbol{Q}_h) $ are the solution of  \eqref{Equ:3.16}-\eqref{Equ:3.17}.
\end{lemma}

Ayuso de Dios, Hiptmair and Pagliantini proved the well-posedness of \eqref{Equ:3.20} in section 2 of \cite{AyusoHiptmair17:646}.
Therefore, by combining Lemma \ref{Lem:0}, we obtain the well-posedness of discrete variational problems \eqref{Equ:3.16}-\eqref{Equ:3.17}.

\subsection{Adaptive Mixed Interior Penalty  Discontinuous Galerkin method(\textbf{AMIPDG})}

Our adaptive cycle can be implemented by the following algorithm:
\begin{algorithm}  	
	\caption{Adaptive Mixed Interior Penalty  Discontinuous Galerkin Method (AMIPDG) cycle} \label{ALG1}
	{\textbf{Input} initial triangulation $\mathcal{T}_0$; data $\boldsymbol{f}$; tolerance tol; marking parameter $\theta\in(0,1)$.}
	
	{\textbf{Output} a triangulation $\mathcal{T}_J$; MIPDG solution $(\boldsymbol{u}_J,\boldsymbol{p}_J)$.}
	
	{$\eta=1;k=0;$}
	
	{\textbf{while} $\eta\geq tol$}
	
	{~~~~\textbf{SOLVE} solve  discrete varational problem \eqref{Equ:3.16}-\eqref{Equ:3.17} on $\mathcal{T}_k$ to get the solution $(\boldsymbol{u}_k,\boldsymbol{p}_k)$;}
	
	{~~\textbf{ESTIMATE} compute the posterior error estimator $\eta=\eta(\boldsymbol{u}_k,\boldsymbol{p}_k,\mathcal{T}_k)$ by using \eqref{Equ:3.24};}
	
	{~~~\textbf{MARK} seek a minimum cardinality $\mathcal{M}_{k} \subset \mathcal{T}_{k}$ such that
		\begin{equation*}\label{mark}
		\eta^{2}\left(\boldsymbol{u}_{k},\boldsymbol{p}_k, \mathcal{M}_{k}\right) \geq \theta \eta^{2}\left(\boldsymbol{u}_{k},\boldsymbol{p}_k, \mathcal{T}_{k}\right);
		\end{equation*}}	
	{~~\textbf{REFINE} bisect elements in $\mathcal{M}_k$ and the neighboring elements to form a conforming $\mathcal{T}_{k+1}$;}
	
	{	~~~~$k=k+1$; }
	
	{\textbf{end}}
	
	{$\boldsymbol{u}_J=\boldsymbol{u}_k;~\boldsymbol{p}_J=\boldsymbol{p}_k;~\mathcal{T}_J=\mathcal{T}_k;$}
\end{algorithm}

Next, we will discuss each step in AEFEM in detail.

\subsubsection{Procedure SOLVE}
For  $ \boldsymbol{f}\in \boldsymbol{L}^2(\Omega)$, and a shape regular mesh $\mathcal{T}_k$, Let $(\boldsymbol{u}_k,\boldsymbol{p}_k)$ be the exact MIPDG solution of \eqref{Equ:3.16}-\eqref{Equ:3.17}. Here, we assume that the solutions $(\boldsymbol{u}_k,\boldsymbol{p}_k)$ can be solved accurately.

\subsubsection{Procedure ESTIMATE}
A posteriori error indicator is an essential ingredient of adaptivity. They are computable quantities depending on the computed solution(s) and data that provide information about the quality of approximation and may consequently be used to make judicious mesh modifications.
Here, we design a new posteriori error estimation indicator for equations \eqref{Equ:3.16}-\eqref{Equ:3.17}, which is similar to that in \cite{XingZhong12}.
For $\tau\in \mathcal{T}_h$, $f \in \mathcal{F}_h$ and $(\boldsymbol{v}_h, \boldsymbol{q}_h) \in \boldsymbol{U}_h\times \boldsymbol{Q}_h$, the residual a posteriori error estimator for the symmetric AMIPDG method is given by
\begin{eqnarray}\nonumber
\eta^2( \boldsymbol{v}_h, \boldsymbol{q}_h;\tau) :&=& \|R_1 (\boldsymbol{v}_h, \boldsymbol{q}_h) \|^2_{L^2(\tau)}
+
h^2_\tau\big(\|R_2 (\boldsymbol{v}_h, \boldsymbol{q}_h)\|^2_{L^2(\tau)} +\|R_3(\boldsymbol{v}_h)\|^2_{L^2(\tau)} \big)\\ \label{Equ:3.23}
&& +\sum_{f\in \partial\tau} h_f \big(\| J_1(\boldsymbol{q}_h)\|^2_{L^2(f)}+ \|J_2(\boldsymbol{v}_h)\|^2_{L^2(f)}\big). 
\end{eqnarray}
They consist of the element residuals and face jump residuals as
\begin{eqnarray*}
	&&  R_1 (\boldsymbol{v}_h, \boldsymbol{q}_h)\vert_\tau := \boldsymbol{q}_h\vert_\tau- \mu\nabla \times \boldsymbol{v}_h\vert_\tau, \\
	&&  R_2 (\boldsymbol{v}_h, \boldsymbol{q}_h)\vert_\tau := \boldsymbol{f}\vert_\tau- (\nabla \times  \boldsymbol{q}_h+  \kappa\boldsymbol{v}_h)\vert_\tau, \\
	&&  R_3(\boldsymbol{v}_h)\vert_\tau := \nabla\cdot (\boldsymbol{f}\vert_\tau- \kappa\boldsymbol{v}_h\vert_\tau), \\
	&&  J_1(\boldsymbol{q}_h)\vert_f :=[[\boldsymbol{q}_h]] , \\
	&&  J_2(\boldsymbol{v}_h)\vert_f :=[[ (\boldsymbol{f}-\kappa\boldsymbol{v}_h)]].
\end{eqnarray*}
where $h_f$ denote the diameter of the circumcircle of $f$, and $h_\tau \approx h_f$.

For any set $ \mathcal{T}'_h  \subseteq \mathcal{T}_h$, the error indicator is defined as
\begin{equation}\label{Equ:3.24}
\eta^2( \boldsymbol{v}_h, \boldsymbol{q}_h;\mathcal{T}'_h)=\sum_{\tau\in\mathcal{T}'_h} \eta^2( \boldsymbol{v}_h, \boldsymbol{q}_h;\tau).
\end{equation}

\subsubsection{Procedure MARK}
We use the D\"{o}rfler mark which was proposed by D\"{o}rfler \cite{Dor96:1106}.
Set marking parameter $\theta \in(0, 1)$, the module MARK outputs a subset of marked elements $\mathcal{M}_k \subset \mathcal{T}_{k} $ with minimal cardinality, such that
\begin{eqnarray}\label{Equ:4.1}
\eta^2( \boldsymbol{v}_k, \boldsymbol{q}_k;\mathcal{M}_k) \geq \theta \eta^2( \boldsymbol{v}_k, \boldsymbol{q}_k;\mathcal{T}_k).
\end{eqnarray}

\subsubsection{Procedure REFINE}
Our implementation of \textbf{REFINE} uses the longest edge bisection strategy. A detailed introduction about the longest edge bisection strategy was provided in \cite{ChenLiFEM}.
To avoid confusion, the relationship between the two tetrahedral
meshes $\mathcal{T}_h$ and $\mathcal{T}_{H}$ that are nested into each other is defined as: $\mathcal{T}_h$ is the new mesh division of $\mathcal{T}_{H}$ after one cycle of the above cycle process, abbreviated as $ \mathcal{T}_H \leq \mathcal{ T}_h$.

\section{Convergence of \textbf{AMIPDG} algorithm  }
In this section, we establish the upper bound estimate of the error. Subsequently, we demonstrate that the sum of the energy error and the error estimator between two consecutive adaptive loops is a contraction. Finally, we proof that the AMIPDG is convergence.

\subsection{The upper bound estimate of the error}
In this subsection, before establishing the reliability of a posteriori error estimator, we need to define the corresponding DG norm, for any $(\boldsymbol{v}, \boldsymbol{q}) \in \boldsymbol{U}\times \boldsymbol{Q}  $ and $(\boldsymbol{v}_h, \boldsymbol{q}_h) \in  \boldsymbol{U}_h\times \boldsymbol{Q}_h$,
\begin{eqnarray}\nonumber
\|(\boldsymbol{v}, \boldsymbol{q})&-&(\boldsymbol{v}_h, \boldsymbol{q}_h)\|^2_{DG}:=
\|\boldsymbol{q}-\boldsymbol{q}_h\|^2_{L^2(\Omega)}
+
\|\kappa(\boldsymbol{v}-\boldsymbol{v}_h)\|^2_{L^2(\Omega)} \\ \label{Equ:3.25}
&+&
\sum\limits_{\tau\in\mathcal{T}_h} \|\mu\nabla\times(\boldsymbol{v}-\boldsymbol{v}_h)\|^2_{L^2(\tau)}
+
\sum\limits_{f\in\mathcal{F}_h}\alpha h_f^{-1} <[[ \boldsymbol{v}_h ]]   , [[ \boldsymbol{v}_h ]]  >_{f}.
\end{eqnarray}

\begin{remark}\label{Rem:3.3}
	For any  $\boldsymbol{v}\in \boldsymbol{U}$ and $\boldsymbol{v}_h\in\boldsymbol{U}_h$, we have
	\begin{equation*}
	\|[[ \boldsymbol{v}_h ]]  \|^2_{L^2(f)} = \|[[  ( \boldsymbol{v}-\boldsymbol{v}_h) ]] \|^2_{L^2(f)},  \quad \forall f\in \mathcal{F}_h.
	\end{equation*}
	In fact, $\boldsymbol{v}\in \boldsymbol{U}$ implies that $[[ \boldsymbol{v} ]] \vert_f=0$ (see Chapter 5  of \cite{Monk03Book}).
\end{remark}

We summarize our main result  in this subsection as follows.
\begin{theorem}\label{Thm:1}
	Let $(\boldsymbol{u}, \boldsymbol{p})\in \mathbf {U}\times \boldsymbol{Q}$ and $(\boldsymbol{u}_h, \boldsymbol{p}_h) \in \boldsymbol{U}_h\times \boldsymbol{Q}_h$ be the solutions of  \eqref{Equ:3.6}-\eqref{Equ:3.7}  and   \eqref{Equ:3.16}-\eqref{Equ:3.17}, respectively. Let $\eta( \boldsymbol{u}_h, \boldsymbol{p}_h;\mathcal{T}_h)$ be the residual error indicator of \eqref{Equ:3.24}. Then we have the following  estimate
	\begin{equation}\label{Equ:3.26}
	\|(\boldsymbol{u}, \boldsymbol{p})-(\boldsymbol{u}_h, \boldsymbol{p}_h)\|^2_{DG} \leq C_1 \eta^2( \boldsymbol{u}_h, \boldsymbol{p}_h;\mathcal{T}_h),
	\end{equation}
	where the constant $C_1$ depending on the shape regularity of mesh.
\end{theorem}

Let $(\boldsymbol{u}_h,\boldsymbol{p}_h)\in\boldsymbol{U}_h\times\boldsymbol{Q}_h$ be the solution of \eqref{Equ:3.16}-\eqref{Equ:3.17}, similarly to \cite{CarHop11:13}, we introduce the nonconformity of the MSIPDG method results in some consistency error:
\begin{equation}\label{Equ:3.27}
\zeta:=\min_{\tilde{\boldsymbol{v}}_h\in\boldsymbol{U}} \big(\sum_{\tau\in\mathcal{T}_h}(\|\boldsymbol{u}_h- \tilde{\boldsymbol{v}}_h\|^2_{ {L}^2(\tau)}+\|\nabla\times( \boldsymbol{u}_h- \tilde{\boldsymbol{v}}_h)\|^2_{ {L}^2(\tau)})  \big)^{1/2}.
\end{equation}
We denote that $\tilde{\boldsymbol{u}}_h\in \boldsymbol{U}$ is the unique minimizer of \eqref{Equ:3.27}, namely
\begin{equation}\label{Eqn:zeta}
\tilde{\zeta}=\big(\sum_{\tau\in\mathcal{T}_h}(\|\boldsymbol{u}_h- \tilde{\boldsymbol{u}}_h\|^2_{ {L}^2(\tau)}+\|\nabla\times( \boldsymbol{u}_h- \tilde{\boldsymbol{u}}_h)\|^2_{ {L}^2(\tau)}) \big)^{1/2}.
\end{equation}

\begin{lemma}\label{Lem:3}
	Let $(\boldsymbol{u}, \boldsymbol{p})\in \mathbf {U}\times \boldsymbol{Q}$ and $(\boldsymbol{u}_h, \boldsymbol{p}_h) \in \boldsymbol{U}_h\times \boldsymbol{Q}_h$ be the solutions of  \eqref{Equ:3.6}-\eqref{Equ:3.7}  and   \eqref{Equ:3.16}-\eqref{Equ:3.17}, respectively, let $\tilde{\boldsymbol{u}}_h$ be the unique minimizer of \eqref{Equ:3.27}, then
	\begin{equation*}
	\|(\boldsymbol{u}-\tilde{\boldsymbol{u}}_h, \boldsymbol{p}-\boldsymbol{p}_h)\|_{\boldsymbol{U}\times \boldsymbol{Q}}
	=
	(\|\boldsymbol{u}-\tilde{\boldsymbol{u}}_h \|^2_{curl,\Omega}
	+
	\|\boldsymbol{p}-\boldsymbol{p}_h\|^2_{0})^{1/2}
	\lesssim \|\tilde{\ell}_1\|_{\boldsymbol{Q}^*}+\|\tilde{\ell}_2\|_{\boldsymbol{U}^*},
	\end{equation*}
	where the residuals $\tilde{\ell}_1\in \boldsymbol{Q }^*$ and $\tilde{\ell}_2\in \boldsymbol{U}^*$ defined by
	\begin{eqnarray}\label{Equ:3.28}
	\tilde{\ell}_1(\boldsymbol{q})
	=
	\ell_1(\boldsymbol{q})
	-
	a(\boldsymbol{p}_h, \boldsymbol{q})
	+
	b(\tilde{\boldsymbol{u}}_h, \boldsymbol{q}), \quad \forall q\in \boldsymbol{Q}, \\\label{Equ:3.29}
	\tilde{\ell}_2(\boldsymbol{v})
	=
	\ell_2(\boldsymbol{v})
	-
	d(\boldsymbol{v}, \boldsymbol{p}_h)
	-
	c(\tilde{\boldsymbol{u}}_h, \boldsymbol{v}), \quad \forall \boldsymbol{v}\in \boldsymbol{U} .
	\end{eqnarray}
\end{lemma}

\begin{proof}
	For any $ \boldsymbol{q}_1, \boldsymbol{q}_2 , \boldsymbol{q}\in \boldsymbol{Q} $ and any $\boldsymbol{v}_1, \boldsymbol{v}_2 , \boldsymbol{v}\in \boldsymbol{U}$. we have the following property  by \eqref{Equ:3.13}
	\begin{eqnarray*}
		\lefteqn
		{
			(\mathcal{A}(\boldsymbol{v}_1+\boldsymbol{v}_2, \boldsymbol{q}_1+\boldsymbol{q}_2))(\boldsymbol{v}, \boldsymbol{q})
		}\\
		&&
		=
		a(\boldsymbol{q}_1+\boldsymbol{q}_2, \boldsymbol{q})-b(\boldsymbol{v}_1+\boldsymbol{v}_2, \boldsymbol{q})+d(\boldsymbol{v}, \boldsymbol{q}_1+\boldsymbol{q}_2)+c(\boldsymbol{v}_1+\boldsymbol{v}_2, \boldsymbol{v})\\
		&&
		=
		a(\boldsymbol{q}_1, \boldsymbol{q})-b(\boldsymbol{v}_1, \boldsymbol{q})+d(\boldsymbol{v}, \boldsymbol{q}_1)+c(\boldsymbol{v}_1, \boldsymbol{v})\\
		&&\quad
		+
		a(\boldsymbol{q}_2, \boldsymbol{q})-b(\boldsymbol{v}_2, \boldsymbol{q})+d(\boldsymbol{v}, \boldsymbol{q}_2)+c(\boldsymbol{v}_2, \boldsymbol{v})\\
		&&
		=
		(\mathcal{A}(\boldsymbol{v}_1, \boldsymbol{q}_1))(\boldsymbol{v}, \boldsymbol{q})+(\mathcal{A}(\boldsymbol{v}_2, \boldsymbol{q}_2))(\boldsymbol{v}, \boldsymbol{q}).
	\end{eqnarray*}
	Thus,
	\begin{eqnarray*}
		\lefteqn{
			(\mathcal{A}(\boldsymbol{u}-\tilde{\boldsymbol{u}}_h, \boldsymbol{p}-\boldsymbol{p}_h))
			(\boldsymbol{v}, \boldsymbol{q})
		}\\
		&&
		=
		(\mathcal{A}(\boldsymbol{u}, \boldsymbol{p}))(\boldsymbol{v}, \boldsymbol{q})-(\mathcal{A}(\tilde{\boldsymbol{u}}_h, \boldsymbol{p}_h))(\boldsymbol{v}, \boldsymbol{q})\\
		&&
		=
		(\ell_1(\boldsymbol{q})+\ell_2(\boldsymbol{v}))- (a(\boldsymbol{p}_h, \boldsymbol{q})-b(\tilde{\boldsymbol{u}}_h, \boldsymbol{q})+d(\boldsymbol{v}, \boldsymbol{p}_h)+c(\tilde{\boldsymbol{u}}_h,\boldsymbol{v}))\\
		&&
		=
		\tilde{\ell}_1(\boldsymbol{q}) +\tilde{\ell}_2(\boldsymbol{v}).
	\end{eqnarray*}
	In fact that $ (\boldsymbol{u}-\tilde{\boldsymbol{u}}_h, \boldsymbol{p}-\boldsymbol{p}_h)\in\boldsymbol{U}\times \boldsymbol{Q} $  and combining the Lemma \ref{Lem:3.0} can  concludes the proof.	
	
\end{proof}

Next, we will provide upper bounds for $\| \tilde{\ell}_1\|_{\boldsymbol{Q }^*}$ and  $\| \tilde{\ell}_2\|_{\boldsymbol{U}^*}$ in Lemmas \ref{Lem:4} and \ref{Lem:5}, respectively.
\begin{lemma}\label{Lem:4}
	Let $(\boldsymbol{u}_h, \boldsymbol{p}_h)\in \boldsymbol{U}_h\times  \boldsymbol{Q}_h$ be the solutions of \eqref{Equ:3.16}-\eqref{Equ:3.17}, and $\tilde{\boldsymbol{u}}_h$ be the unique minimizer of \eqref{Equ:3.27}. Then we get the estimate of the linear functional $\tilde{\ell}_1$ defined in \eqref{Equ:3.28} as following
	\begin{equation}\label{Eqn:l1t}
	\| \tilde{\ell}_1\|_{\boldsymbol{Q }^*}
	\lesssim
	\big( \sum_{\tau\in\mathcal{T}_h} \|R_1( \boldsymbol{u}_h, \boldsymbol{p}_h)\|^2_{L^2(\tau)} \big)^{1/2}
	+
	\big( \sum_{\tau\in\mathcal{T}_h}\|\nabla\times (\tilde{\boldsymbol{u}}_h-\boldsymbol{u}_h)\|^2_{L^2(\tau)}\big)^{1/2}.
	\end{equation}
\end{lemma}

\begin{proof}
	For any $\boldsymbol{q}\in\boldsymbol{Q}$,
	by the definition of $\tilde{\ell}_1$, we have
	\begin{equation*}
	\tilde{\ell}_1(\boldsymbol{q})
	=
	\sum_{\tau\in\mathcal{T}_h}\int_\tau \big((\mu\nabla\times \boldsymbol{u}_h- \boldsymbol{p}_h)
	+
	\mu\nabla\times (\tilde{\boldsymbol{u}}_h-\boldsymbol{u}_h)\big)\cdot \boldsymbol{q} \mathrm{d} \boldsymbol{x}.
	\end{equation*}
	Then applying the H\"{o}lder inequality and the Cauchy-Schwarz inequality,
	\begin{eqnarray*}
		\lefteqn{\vert \tilde{\ell}_1(\boldsymbol{q})\vert
			\leq
			\sum_{\tau\in\mathcal{T}_h } \| \mu\nabla\times \boldsymbol{u}_h- \boldsymbol{p}_h\|_{L^2(\tau)}\|\boldsymbol{q}\|_{L^2(\Omega)}
			+
			\sum_{\tau\in\mathcal{T}_h }\|\mu\nabla\times (\tilde{\boldsymbol{u}}_h-\boldsymbol{u}_h)\|_{L^2(\tau)} \|\boldsymbol{q}\|_{L^2(\Omega)}}\\
		& \lesssim &
		\bigg(\big( \sum_{\tau\in\mathcal{T}_h} \|R_1( \boldsymbol{u}_h, \boldsymbol{p}_h)\|^2_{L^2(\tau)} \big)^{1/2}
		+
		\big( \sum_{\tau\in\mathcal{T}_h  }\|\nabla\times (\tilde{\boldsymbol{u}}_h-\boldsymbol{u}_h)\|^2_{L^2(\tau)}\big)^{1/2}\bigg)\|\boldsymbol{q}\|_{L^2(\Omega)},
	\end{eqnarray*}
	conclude the proof.
\end{proof}

Before estimating the term $\|\tilde{\ell}_2\|_{\boldsymbol{U}^*}$, we need to introduce the following interpolation operator with  the corresponding approximations.
\begin{lemma}\label{Lem:4.1} [\cite{Sch08:633}, Theorem 1]
	Let  $Nd^1_0(\Omega; \mathcal{T}_h)$ be the lowest order edge elements of N\'ed\'elec first family. Then there exists
	an operator $\Pi_h: \boldsymbol{H}_0(curl;\Omega) \to Nd^1_0(\Omega; \mathcal{T}_h)$ with the following properties: For every $\boldsymbol{v} \in \boldsymbol{H}_0(curl;\Omega)$, there exist $\varphi \in H_0^1(\Omega) $ and $   \boldsymbol{z}\in \boldsymbol{H}_0^1(\Omega)$, such that
	\begin{eqnarray*}
		\boldsymbol{v} -\Pi_h \boldsymbol{v} =\nabla \varphi+ \boldsymbol{z}.
	\end{eqnarray*}
	And	for any  $\tau \in \mathcal{T}_h$ and $f \in \mathcal{F}_h$, we have
	\begin{eqnarray*}
		h_\tau^{-1}  \| \varphi\|_{L^2(\tau)} + \|\nabla \varphi\|_{L^2(\tau)}  \lesssim h_\tau \| \boldsymbol{v}\|_{L^2(\Omega_\tau)}, \\
		h_\tau^{-1}  \| \boldsymbol{z} \|_{L^2(\tau)} + \| \nabla \boldsymbol{z}\|_{L^2(\tau)}  \lesssim h_\tau \|\nabla\times \boldsymbol{v}\|_{L^2(\Omega_\tau)},
	\end{eqnarray*}
	where $ \Omega_\tau=\bigcup\limits_{f\in\tau} \Omega_f$,  $\Omega_f=\{ \tau' \in \mathcal{T}_h, f\in \tau'  \} $, and the constants depending on the shape regularity of the mesh.
\end{lemma}

\begin{lemma}\label{Lem:5}
	Let  $ (\boldsymbol{u}_h, \boldsymbol{p}_h)\in \boldsymbol{U}_h\times  \boldsymbol{Q}_h $ be the  solution of \eqref{Equ:3.16}-\eqref{Equ:3.17}, and $\tilde{\boldsymbol{u}}_h$  be the unique solution of \eqref{Equ:3.27}.
	Then the linear functional $\tilde{\ell}_2$ defined in \eqref{Equ:3.29} satisfies the following estimate
	\begin{eqnarray}\nonumber
	\lefteqn{
		\| \tilde{\ell}_2\|_{\boldsymbol{U}^*}
		\lesssim
		\bigg( \sum_{\tau\in\mathcal{T}} h^2_\tau(\|R_2( \boldsymbol{u}_h, \boldsymbol{p}_h)\|^2_{L^2(\tau)}
		+
		\|R_2( \boldsymbol{u}_h)\|^2_{L^2(\tau)})
	}
	\\  \label{Eqn:l2t}
	&&\ \ \ \
	+
	\sum_{f\in\mathcal{F} }h_f(\|J_1(\boldsymbol{p}_h)\|^2_{L^2(f)}
	+
	\|J_2(\boldsymbol{u}_h)\|^2_{L^2(f)})
	+
	\sum_{\tau\in\mathcal{T}}\| \boldsymbol{u}_h-\tilde{\boldsymbol{u}}_h \|^2_{L^2(\tau)}\bigg)^{1/2}.\ \
	\end{eqnarray}
\end{lemma}
\begin{proof}
	For any $\boldsymbol{v}\in\boldsymbol{U}$ and $\Pi_h$ given by Lemma \ref{Lem:4.1}, we have
	\begin{eqnarray}\label{Eqn:dec}
	\boldsymbol{v} -\Pi_h \boldsymbol{v} =\nabla \varphi+ \boldsymbol{z},
	\end{eqnarray}
	where $\varphi \in H_0^1(\Omega) $ and $   \boldsymbol{z}\in \boldsymbol{H}_0^1(\Omega)$.
	According to linearity of the operator $\tilde{\ell}_2$ and \eqref{Eqn:dec}, we have
	\begin{eqnarray} \label{Equ:3.31}
	\tilde{\ell}_2(\boldsymbol{v})
	=
	\tilde{\ell}_2(\Pi_h\boldsymbol{v})+\tilde{\ell}_2(\boldsymbol{v}-\Pi_h\boldsymbol{v})
	=
	\tilde{\ell}_2(\Pi_h\boldsymbol{v})
	+
	\tilde{\ell}_2( \nabla \varphi)+\tilde{\ell}_2( \boldsymbol{z} ).
	\end{eqnarray}
	We will next estimate the three terms on the right hand side of \eqref{Equ:3.31}.
	
	For the first term $\tilde{\ell}_2(\Pi_h\boldsymbol{v})$ of \eqref{Equ:3.31}, using the definition of $\tilde{\ell}_2$, we have
	\begin{eqnarray*}
		\tilde{\ell}_2(\Pi_h\boldsymbol{v})
		&=&
		\ell_2(\Pi_h\boldsymbol{v})
		-
		d(\Pi_h\boldsymbol{v}, \boldsymbol{p}_h)
		-
		c(\tilde{\boldsymbol{u}}_h, \Pi_h\boldsymbol{v})\\
		&=&
		\ell_2(\Pi_h\boldsymbol{v})
		-
		d(\Pi_h\boldsymbol{v}, \boldsymbol{p}_h)
		-
		c(\boldsymbol{u}_h, \Pi_h\boldsymbol{v})
		+
		c(\boldsymbol{u}_h-\tilde{\boldsymbol{u}}_h, \Pi_h\boldsymbol{v}).
	\end{eqnarray*}
	Noting that $\Pi_h\boldsymbol{v}\in Nd^1_0(\Omega;\mathcal{T}_h)\subseteq \boldsymbol{U}_h$ has zero jumps, and combining \eqref{Equ:3.17}, we have
	\begin{eqnarray*}
		\ell_2(\Pi_h\boldsymbol{v})
		-
		d(\Pi_h\boldsymbol{v}, \boldsymbol{p}_h)
		-
		c(\boldsymbol{u}_h, \Pi_h\boldsymbol{v})
		=
		\ell_{2,h}(\Pi_h\boldsymbol{v})
		-
		d_h(\Pi_h\boldsymbol{v}, \boldsymbol{p}_h)
		-
		c_h(\boldsymbol{u}_h, \Pi_h\boldsymbol{v})
		=0.
	\end{eqnarray*}
	Thus, we have
	\begin{eqnarray*}
		\tilde{\ell}_2(\Pi_h\boldsymbol{v})
		&=&
		c(\boldsymbol{v}_h-\tilde{\boldsymbol{u}}_h, \Pi_h\boldsymbol{v})\\
		&=&
		c(\boldsymbol{v}_h-\tilde{\boldsymbol{u}}_h, \boldsymbol{v})
		+
		c(\boldsymbol{v}_h-\tilde{\boldsymbol{u}}_h, \Pi_h\boldsymbol{v}-\boldsymbol{v})\\
		&\leq&
		\|\kappa \|_{0,\infty} \|\boldsymbol{v}_h-\tilde{\boldsymbol{u}}_h\|_{0, \mathcal{T}_h}( \|\boldsymbol{v}\|_{0, \mathcal{T}_h}
		+
		\|\Pi_h\boldsymbol{v}-\boldsymbol{v}\|_{0, \mathcal{T}_h}).
	\end{eqnarray*}
	Then using \eqref{Eqn:dec}, triangle inequality and Lemma \ref{Lem:4.1}, we get
	\begin{eqnarray}\nonumber
	\tilde{\ell}_2(\Pi_h\boldsymbol{v})
	&\leq&
	\|\kappa\|_{0,\infty}\|\boldsymbol{v}_h-\tilde{\boldsymbol{u}}_h\|_{0, \mathcal{T}_h}( \|\boldsymbol{v}\|_{0, \mathcal{T}_h}
	+
	\|\nabla \varphi+ \boldsymbol{z}\|_{0, \mathcal{T}_h})\\ \nonumber
	&\leq&
	\|\kappa \|_{0,\infty}  \| \boldsymbol{v}_h-\tilde{\boldsymbol{u}}_h\|_{0, \mathcal{T}_h}( \|\boldsymbol{v}\|_{0, \mathcal{T}_h} +\|\nabla \varphi\|_{0, \mathcal{T}_h}+  \|\boldsymbol{z}\|_{0, \mathcal{T}_h} ) \\\label{Equ:3.32}
	&\leq&
	\|\kappa \|_{0,\infty} \| \boldsymbol{v}_h-\tilde{\boldsymbol{u}}_h\|_{0, \mathcal{T}_h} \|\boldsymbol{v}\|_{curl,\mathcal{T}_h}.
	\end{eqnarray}
	
	For the second term $\tilde{\ell}_2(\nabla\varphi)$ of \eqref{Equ:3.31}, using the definition of $\tilde{\ell}_2$, \eqref{Equ:3.12}, \eqref{Equ:3.9}, \eqref{Eqn:3.10} and the fact $\nabla\times\nabla\varphi = 0$, which implies
	\begin{eqnarray}\nonumber
	\tilde{\ell}_2( \nabla \varphi)
	&=&\ell_2(  \nabla \varphi )-d( \nabla \varphi, \boldsymbol{p}_h)-c(\tilde{\boldsymbol{u}}_h, \nabla \varphi)\\ \nonumber
	&=&(\boldsymbol{f}, \nabla \varphi)- (\nabla\times  \nabla \varphi, \boldsymbol{p}_h )
	-(\kappa\tilde{\boldsymbol{u}}_h, \nabla \varphi )\\
	\label{Eqn:var}
	&=&(\boldsymbol{f}, \nabla \varphi)-(\kappa\tilde{\boldsymbol{u}}_h, \nabla \varphi ).
	\end{eqnarray}
	By \eqref{Eqn:var} and Green's formula, we have
	\begin{eqnarray*}
		\tilde{\ell}_2( \nabla \varphi)
		&=&(\boldsymbol{f}, \nabla \varphi)-  (\kappa\boldsymbol{u}_h, \nabla \varphi )
		+
		( \kappa(\boldsymbol{u}_h- \tilde{\boldsymbol{u}}_h), \nabla \varphi )\\
		& \leq &
		\sum\limits_{\tau\in \mathcal{T}_h}(R_3(\boldsymbol{u}_h), \varphi)_{0, \tau} +\sum\limits_{f\in \mathcal{F}_h} <J_2(\boldsymbol{u}_h), \varphi>_{0, f}
		+( \kappa(\boldsymbol{u}_h- \tilde{\boldsymbol{u}}_h), \nabla \varphi ).
	\end{eqnarray*}
	Applying the Cauchy-Schwarz inequality, Lemma \ref{Lem:4.1} and  trace inequality, we have
	\begin{eqnarray}\nonumber
	\tilde{\ell}_2 ( \nabla \varphi)
	\leq
	\bigg(\sum\limits_{\tau\in \mathcal{T}_h} h_\tau^2\|R_3(\boldsymbol{u}_h)\|^2_{0, \tau}
	+\sum\limits_{f\in \mathcal{F}_h}h_f \|J_2(\boldsymbol{u}_h)\|^2_{0, f} \bigg.
	\\ \label{Equ:3.34}
	\bigg.
	+
	\sum\limits_{\tau\in \mathcal{T}_h}\|\kappa \|_{0,\infty}  \|\boldsymbol{u}_h- \tilde{\boldsymbol{u}}_h\|^2_{0, \tau}
	\bigg)^{1/2} \|\boldsymbol{v}\|_{curl, \mathcal{T}_h }.
	\end{eqnarray}
	
	Similarly, for the third term $\tilde{\ell}_2(\boldsymbol{z})$ of \eqref{Equ:3.31}, we have
	\begin{eqnarray} \nonumber
	\tilde{\ell}_2( \boldsymbol{z} )
	&=&
	(\boldsymbol{f}, \boldsymbol{z}  )
	-
	(\nabla\times   \boldsymbol{z} , \boldsymbol{p}_h )
	-
	(\kappa\tilde{\boldsymbol{u}}_h, \boldsymbol{z}  )\\ \nonumber
	&=&
	(\boldsymbol{f}, \boldsymbol{z}  )
	-
	(\nabla\times   \boldsymbol{z} , \boldsymbol{p}_h )
	-
	(\kappa\boldsymbol{u}_h, \boldsymbol{z}  )
	+
	(  \kappa(\boldsymbol{u}_h-\tilde{\boldsymbol{u}}_h), \boldsymbol{z}  )\\\nonumber
	&\leq&
	\bigg(\sum\limits_{\tau\in \mathcal{T}_h} h_\tau^2\|R_2(\boldsymbol{u}_h,\boldsymbol{p}_h)\|^2_{0, \tau}
	+
	\sum\limits_{f\in \mathcal{F}_h}h_f \|J_1(\boldsymbol{p}_h)\|^2_{0, f} \bigg. \\\label{Equ:3.35}
	\bigg.
	&&
	+
	\sum\limits_{\tau\in \mathcal{T}_h}\|\kappa \|_{0,\infty} \|\boldsymbol{u}_h- \tilde{\boldsymbol{u}}_h\|^2_{0, \tau}\bigg)^{1/2} \|\boldsymbol{v}\|_{curl, \mathcal{T}_h }.
	\end{eqnarray}
	
	Substituting \eqref{Equ:3.32}, \eqref{Equ:3.34} and \eqref{Equ:3.35} into \eqref{Equ:3.31}, the proof is completed.
\end{proof}

Notice that both \eqref{Eqn:l1t} and \eqref{Eqn:l2t} are related to the terms $\sum\limits_{\tau\in\mathcal{T}_h}\|\nabla\times (\tilde{\boldsymbol{u}}_h-\boldsymbol{u}_h)\|^2_{L^2(\tau)}$ and $\sum\limits_{\tau\in\mathcal{T}}\| \boldsymbol{u}_h-\tilde{\boldsymbol{u}}_h \|^2_{L^2(\tau)}$, which are a part of $\tilde{\zeta}$. Therefore, we  prove upper bounds for $\tilde{\zeta}$ in the following Lemma.

\begin{lemma}\label{Lem:3.2}
	Let $({\boldsymbol{u}}_h, {\boldsymbol{p}}_h) \in \boldsymbol{U}_h\times \boldsymbol{Q}_h$ be the solutions of \eqref{Equ:3.16}-\eqref{Equ:3.17} and $\tilde{\zeta}$ be consistency error of \eqref{Eqn:zeta}, we have
	\begin{equation}\label{Equ:3.30}
	\tilde{\zeta}^2
	\lesssim
	\eta^2( \boldsymbol{u}_h, \boldsymbol{p}_h;\mathcal{T}_h).
	\end{equation}
\end{lemma}

\begin{proof}
	For any $\boldsymbol{v}_h \in \boldsymbol{U}_h $, there exit an interpolation operator  $\mathcal{I}_h :\boldsymbol{H}^1(\Omega;\mathcal{T}_h )\to \boldsymbol{U}^c_h$, such that(see Proposition 4.5 of \cite{HouPer05:485})
	\begin{eqnarray}\label{Eqn:uhIuh}
	\|\boldsymbol{v}_h- \mathcal{I}_h\boldsymbol{v}_h\|^2_{L^2(\Omega)} \lesssim \sum_{f \in \mathcal{F}_h}   h_f \| [[ \boldsymbol{v}_h]] \|^2_{L^2(f)}, \\ \label{Eqn:nablauhIu}
	\sum_{ \tau \in \mathcal{T}_h} \|\nabla \times (\boldsymbol{v}_h-\mathcal{I}_h \boldsymbol{v}_h) \|^2_{L^2(\tau)} \lesssim\sum_{f \in \mathcal{F}_h}  h_f^{-1} \| [[ \boldsymbol{v}_h]] \|^2_{L^2(f)}.
	\end{eqnarray}
	
	Then, combining \eqref{Equ:3.27}, \eqref{Eqn:zeta}, \eqref{Eqn:uhIuh}, \eqref{Eqn:nablauhIu}, and the fact $h_f<1$, we get
	\begin{eqnarray}\nonumber
	\tilde{\zeta}^2
	&=&
	\sum_{\tau\in\mathcal{T}_h}(\|\boldsymbol{u}_h- \tilde{\boldsymbol{u}}_h\|^2_{ {L}^2(\tau)}
	+
	\|\nabla\times( \boldsymbol{u}_h- \tilde{\boldsymbol{u}}_h)\|^2_{ {L}^2(\tau)})  \\ \nonumber
	\label{Eqn:7.1}
	&\leq&
	\sum_{\tau\in\mathcal{T}_h}(\|\boldsymbol{u}_h - \mathcal{I}_h \boldsymbol{u}_h\|^2_{ {L}^2(\tau)}
	+
	\|\nabla\times( \boldsymbol{u}_h - \mathcal{I}_h \boldsymbol{u}_h)\|^2_{ {L}^2(\tau)}) \\\nonumber
	&\lesssim& \sum_{f \in \mathcal{F}_h}   h_f \| [[ \boldsymbol{u}_h]] \|^2_{L^2(f)}
	+
	\sum_{f \in \mathcal{F}_h}   h_f^{-1} \| [[ \boldsymbol{u}_h]] \|^2_{L^2(f)}  \\ \label{Eqn:7.2}
	&\lesssim& \sum_{f \in \mathcal{F}_h}   h_f^{-1} \| [[ \boldsymbol{u}_h]] \|^2_{L^2(f)}.
	\end{eqnarray}
	
	Noting that $(\boldsymbol{u}_h, \boldsymbol{p}_h) \in  \boldsymbol{U}_h\times \boldsymbol{Q}_h $ is the solution of discrete variational problem \eqref{Equ:3.16}-\eqref{Equ:3.17}.
	Then by using Lemma \ref{Lem:0}, we know that $\boldsymbol{u}_h$ is the solution of discrete variational problem \eqref{Equ:3.20}. Hence, we have ( see Lemma 5 of \cite{XingZhong12})
	\begin{eqnarray}\label{Eqn:ajump}
	\alpha\| h_f^{-1/2}[[ \boldsymbol{u}_h]] \|_{L^2(\mathcal{F}_h)} \lesssim  \eta( \boldsymbol{u}_h, \boldsymbol{p}_h;\mathcal{T}_h).
	\end{eqnarray}
	
	At last, combining \eqref{Eqn:7.2} and \eqref{Eqn:ajump}, we have
	\begin{eqnarray*}	
		\tilde{\zeta}^2	&  \lesssim &\eta^2( \boldsymbol{u}_h, \boldsymbol{p}_h;\mathcal{T}).
	\end{eqnarray*}
\end{proof}

Combining Lemmas \ref{Lem:3}, \ref{Lem:4}, \ref{Lem:5} and \ref{Lem:3.2}, we will  prove Theorem \ref{Thm:1}.
\begin{proof}[ \textbf{Proof of} Theorem \ref{Thm:1}:]
	By using \eqref{Equ:3.25}, the triangle inequality, \eqref{Eqn:zeta},  Lemmas \ref{Lem:3}, \ref{Lem:4}, \ref{Lem:5}, \ref{Lem:3.2} and \eqref{Eqn:ajump}, we get
	\begin{eqnarray*}\lefteqn{
			\|(\boldsymbol{u}, \boldsymbol{p})-(\boldsymbol{u}_h, \boldsymbol{p}_h)\|^2_{DG} }\\
		&\lesssim&
		\|\boldsymbol{p}-\boldsymbol{p}_h\|^2_{L^2(\Omega)}
		+
		\|\kappa(\boldsymbol{u}-\boldsymbol{u}_h)\|^2_{L^2(\Omega)}
		\\
		&&
		+
		\sum\limits_{\tau\in\mathcal{T}_h}\|\nabla\times\mu({\boldsymbol{u}}-\boldsymbol{u}_h)\|^2_{L^2(\tau)}
		+
		\sum\limits_{f\in\mathcal{F}_h}\alpha h_f^{-1} <[[ \boldsymbol{u}_h ]]   , [[ \boldsymbol{u}_h ]]  >_{f}\\
		&\lesssim&
		\|\boldsymbol{p}-\boldsymbol{p}_h\|^2_{L^2(\Omega)}
		+
		\|\boldsymbol{u}-\tilde{\boldsymbol{u}}_h\|^2_{\boldsymbol{curl}, \Omega}
		+
		\tilde{\zeta}^2
		+
		\sum\limits_{f\in\mathcal{F}_h}\alpha h_f^{-1}  <[[ \boldsymbol{u}_h ]]   , [[ \boldsymbol{u}_h ]]  >_{f}   \\
		&=&
		\|(\boldsymbol{u}-\tilde{\boldsymbol{u}}_h, \boldsymbol{p}-\boldsymbol{p}_h)\|_{\boldsymbol{U}\times \boldsymbol{Q}}
		+
		\tilde{\zeta}^2
		+
		\sum\limits_{f\in\mathcal{F}_h}\alpha h_f^{-1}  <[[ \boldsymbol{u}_h ]]   , [[ \boldsymbol{u}_h ]]  >_{f}  \\
		&\lesssim&
		\| \tilde{\ell}_1\|^2_{\mathbf {Q}^*}
		+
		\|\tilde{\ell}_2\|^2_{\boldsymbol{U}^*}
		+
		\tilde{\zeta}^2
		+
		\sum\limits_{f\in\mathcal{F}_h}\alpha h_f^{-1} <[[ \boldsymbol{u}_h ]]   , [[ \boldsymbol{u}_h ]]  >_{f} \\
		&\leq&    C_1 \eta^2( \boldsymbol{u}_h, \boldsymbol{p}_h;\mathcal{T}_h).
	\end{eqnarray*}
\end{proof}

\subsection{The error reduces on two successive meshes}

For convenience, for any $\boldsymbol{v}\in\boldsymbol{U}$ and $\boldsymbol{v}_h\in\boldsymbol{U}_h$, we denote
\begin{eqnarray}  \nonumber
\|\vert\boldsymbol{v}-\boldsymbol{v}_h \vert\|_h^2
&=&
\|\kappa(\boldsymbol{v}-\boldsymbol{v}_h)\|^2_{L^2(\Omega)}
+
\sum\limits_{\tau\in\mathcal{T}_h} \|\nabla\times\mu(\boldsymbol{v}-\boldsymbol{v}_h)\|^2_{L^2(\tau)} \\\label{Equ:4.0}
&&+
\sum\limits_{f\in\mathcal{F}_h} \alpha h_f^{-1}  <[[  \boldsymbol{v}_h ]]   , [[ \boldsymbol{v}_h ]]  >_{f}.
\end{eqnarray}

Let $\boldsymbol{U}^{c}_h$ be the $\boldsymbol{H}(\boldsymbol{curl})$ conforming subspace  of  $\boldsymbol{U}_h$ given by
$$
\boldsymbol{U}^{c}_h :=  \boldsymbol{U}_h \cap \mathbf {H}_0(curl;\Omega).
$$
Then, there is a subspace $\boldsymbol{U}^{\bot}_h$ which can orthogonally decompose $\boldsymbol{U}_h $ under $\boldsymbol{L}^2$ inner product such that
$\boldsymbol{U}_h :=\boldsymbol{U}^{c}_h \oplus \boldsymbol{U}^{\bot}_h$. Especially, if  $(\boldsymbol{u}_h, \boldsymbol{p}_h) \in  \boldsymbol{U}_h\times \boldsymbol{Q}_h $ is the  solution of  \eqref{Equ:3.16}-\eqref{Equ:3.17}, then we have
\begin{equation}\label{Eqn:bot}
\|\vert\boldsymbol{u}_h^\bot \vert\|^2_{h}
\lesssim
\alpha  \sum\limits_{f \in \partial\tau } \|h_f^{-1/2} [[ \boldsymbol{u}_h]] \|^2_{L^2(f)}.
\end{equation}
In fact, from the Lemma \ref{Lem:0}, notice that $\boldsymbol{u}_h$ satisfies the IPDG scheme of \eqref{Equ:3.20}, and
according to Lemma 2 in \cite{XingZhong12}, we can obtain \eqref{Eqn:bot}.

In order to easily estimate the jump term of face $\mathcal{F}_h$, we need to introduce the lifting operators and the corresponding stability estimates, more details are referenced to Proposition 12 in \cite{PerSch02:4675}.

Let $\mathcal{L}_h:\boldsymbol{H}^1(\Omega;\mathcal{T}_h )\to \boldsymbol{U}_h $ be the lifting operators, which satisfies the following equality
\begin{equation}\label{Eqn:L}
\int_\Omega \mathcal{L}_h(\boldsymbol{v} )\cdot \boldsymbol{w} \mathrm{d}x=<[[ \boldsymbol{v}]] , \{\{  \boldsymbol{w}  \}\}  >_{\mathcal{F}_h}, \quad  \forall \boldsymbol{w} \in \boldsymbol{U}_h,
\end{equation}
and
\begin{equation}\label{Eqn:se}
\| \mathcal{L}_h(\boldsymbol{v} )\|_{L^2(\Omega)} \leq C_{\mathcal{L}}\| h^{-1/2} [[ \boldsymbol{v}]] \|_{L^2(\mathcal{F}_h)},
\end{equation}
where the constant $C_{\mathcal{L}}$ depending on the shape regularity of mesh $\mathcal{T}_h$ and the degree of polynomial $l$.

\begin{lemma}\label{Lem:add1}
	Let $(\boldsymbol{u}, \boldsymbol{p})\in \mathbf {U}\times \boldsymbol{Q}$ and $(\boldsymbol{u}_h, \boldsymbol{p}_h) \in \boldsymbol{U}_h\times \boldsymbol{Q}_h$ be the solutions of  \eqref{Equ:3.6}-\eqref{Equ:3.7}  and   \eqref{Equ:3.16}-\eqref{Equ:3.17}, respectively, we have
	\begin{eqnarray} \label{Equ:add1}
	\|\boldsymbol{p}-\boldsymbol{p}_h \|_{L^2(\Omega)}
	&\lesssim&
	\| \nabla \times ( \boldsymbol{u}-\boldsymbol{u}_h) \|_{L^2(\Omega)}+ \eta( \boldsymbol{u}_{h}, \boldsymbol{p}_{h};\mathcal{T}_{h}),\\\nonumber
	\|\boldsymbol{p}_h-\boldsymbol{p}_H \|_{L^2(\Omega)}
	&\lesssim&
	\| \nabla \times ( \boldsymbol{u}_h-\boldsymbol{u}_H) \|_{L^2(\Omega)}
	\\
	&&+
	\bigg(\eta( \boldsymbol{u}_{h}, \boldsymbol{p}_{h};\mathcal{T}_{h})+ \eta( \boldsymbol{u}_{H}, \boldsymbol{p}_{H};\mathcal{T}_{H})\bigg).
	\end{eqnarray}
\end{lemma}

\begin{proof}
	Noting that $\boldsymbol{Q}_h \subseteq \boldsymbol{Q}$, and using \eqref{Equ:3.6}, the definition of $R_1(\boldsymbol{u}_h, \boldsymbol{p}_h)$ and  \eqref{Equ:3.23}, we have
	\begin{eqnarray*}
		\|\boldsymbol{p}-\boldsymbol{p}_h\|_{L^2(\mathcal{T}_h)}
		&\leq&
		\sup\limits_{\forall \boldsymbol{q} \in  \boldsymbol{Q}   } \frac{(\boldsymbol{p}-\boldsymbol{p}_h, \boldsymbol{q} )_{\mathcal{T}_h}}{\|\boldsymbol{q} \|_{L^2(\mathcal{T}_h)}} \\
		&=&
		\sup\limits_{\forall \boldsymbol{q} \in  \boldsymbol{Q}   } \frac{(\mu\nabla \times \boldsymbol{u}, \boldsymbol{q} )_{\mathcal{T}_h} -\big(R_1(\boldsymbol{u}_h, \boldsymbol{p}_h)+ \mu\nabla \times \boldsymbol{u}_h, \boldsymbol{q}\big)_{\mathcal{T}_h}                }{\|\boldsymbol{q} \|_{L^2(\mathcal{T}_h)}} \\
		&\leq&
		\sup\limits_{\forall \boldsymbol{q} \in  \boldsymbol{Q}  }  \frac{(\mu\nabla \times (\boldsymbol{u}-\boldsymbol{u}_h), \boldsymbol{q})_{\mathcal{T}_h}- \big(R_1(\boldsymbol{u}_h, \boldsymbol{p}_h), \boldsymbol{q}\big)_{\mathcal{T}_h}}{\|\boldsymbol{q}\|_{L^2(\mathcal{T}_h)}} \\
		&\lesssim&
		\| \nabla \times (\boldsymbol{u}-\boldsymbol{u}_h)\|_{L^2(\mathcal{T}_h)}+\eta( u_{h}, \boldsymbol{p}_{h};\mathcal{T}_{h}).
	\end{eqnarray*}
	Similarly,  using the definition of $R_1(\boldsymbol{u}_h, \boldsymbol{p}_h)$,   \eqref{Equ:3.16}, \eqref{Eqn:bot}-\eqref{Eqn:se}, and the fact $[[\boldsymbol{u}_h]] = [[\boldsymbol{u}_h^c + \boldsymbol{u}_h^{\bot}]] = [[\boldsymbol{u}_h^{\bot}]]$, we have
	\begin{eqnarray*}
		\lefteqn{\|\boldsymbol{p}_h-\boldsymbol{p}_H \|_{L^2(\mathcal{T}_h)}
			\leq\sup\limits_{\forall \boldsymbol{q}_h \in \boldsymbol{Q}_h} \frac{(\boldsymbol{p}_h-\boldsymbol{p}_H, \boldsymbol{q}_h )_{\mathcal{T}_h}}{\|\boldsymbol{q}_h \|_{L^2(\mathcal{T}_h)}} }\\
		&\leq&
		\sup\limits_{\forall \boldsymbol{q}_h \in \boldsymbol{Q}_h} \frac{(\boldsymbol{p}_h, \boldsymbol{q}_h )_{\mathcal{T}_h}- \big(R_1(\boldsymbol{u}_H, \boldsymbol{p}_H)+ \mu\nabla \times \boldsymbol{u}_H, \boldsymbol{q}_h\big)_{\mathcal{T}_h} }{\|\boldsymbol{q}_h \|_{L^2(\mathcal{T}_h)}} \\
		&\leq&
		\sup\limits_{\forall \boldsymbol{q}_h \in \boldsymbol{Q}_h}  \frac{(\mu\nabla \times \boldsymbol{u}_h, \boldsymbol{q}_h)_{\mathcal{T}_h}+ < \{\{ \boldsymbol{q}_h  \}\} , [[\mu  \boldsymbol{u}_h ]]  >_{\mathcal{F}_h}- \big(R_1(\boldsymbol{u}_H, \boldsymbol{p}_H)+ \mu\nabla \times \boldsymbol{u}_H, \boldsymbol{q}_h\big)_{\mathcal{T}_h}}{\|\boldsymbol{q}_h \|_{L^2(\mathcal{T}_h)}} \\
		&=&
		\sup\limits_{\forall \boldsymbol{q}_h \in \boldsymbol{Q}_h}  \frac{(\mu\nabla \times (\boldsymbol{u}_h-\boldsymbol{u}_H ), \boldsymbol{q}_h)_{\mathcal{T}_h}+ < \{\{ \boldsymbol{q}_h  \}\} , [[\mu  \boldsymbol{u}_h ]]  >_{\mathcal{F}_h}- \big(R_1(\boldsymbol{u}_H, \boldsymbol{p}_H), \boldsymbol{q}_h\big)_{\mathcal{T}_h} }{\|\boldsymbol{q}_h \|_{L^2(\mathcal{T}_h)}} \\
		& \lesssim&
		\| \nabla \times (\boldsymbol{u}_h-\boldsymbol{u}_H)\|_{L^2(\mathcal{T}_h)} + \| h_\tau^{-1/2}[[  \boldsymbol{u}_h ]]  \|_{L^2(\mathcal{T}_h)}+ \eta( u_{H}, \boldsymbol{p}_{H};\mathcal{T}_{H})\\
		& \lesssim &
		\| \nabla \times (\boldsymbol{u}_h-\boldsymbol{u}_H)\|_{L^2(\mathcal{T}_h)} +
		C_{\mathcal{L}}\| h_\tau^{-1/2}[[  \boldsymbol{u}_h^\bot ]]  \|_{L^2(\mathcal{T}_h)}+\eta( u_{H}, \boldsymbol{p}_{H};\mathcal{T}_{H})
		\\
		&  \lesssim &
		\| \nabla \times ( \boldsymbol{u}_h-\boldsymbol{u}_H) \|_{L^2(\tau)}+ \bigg(\eta( u_{h}, \boldsymbol{p}_{h};\mathcal{T}_{h})+ \eta( u_{H}, \boldsymbol{p}_{H};\mathcal{T}_{H})\bigg) .
	\end{eqnarray*}
\end{proof}
\begin{remark} \label{Rem:4}
	Noting that  $\|(\boldsymbol{u}, \boldsymbol{p})-(\boldsymbol{u}_h, \boldsymbol{p}_h)\|^2_{DG} +\eta^2( \boldsymbol{u}_{h}, \boldsymbol{p}_{h};\mathcal{T}_{h})$ and
	$ \|\vert\boldsymbol{u}-\boldsymbol{u}_h \vert\|_h^2 +\eta^2( \boldsymbol{u}_{h}, \boldsymbol{p}_{h};\mathcal{T}_{h})$ are equivalent. In fact,
	by \eqref{Equ:add1}, we first know that
	\begin{eqnarray*}
		\lefteqn{
			\|(\boldsymbol{u}, \boldsymbol{p})-(\boldsymbol{u}_h, \boldsymbol{p}_h)\|^2_{DG} +\eta^2( \boldsymbol{u}_{h}, \boldsymbol{p}_{h};\mathcal{T}_{h})
		}\\
		&&\quad
		=
		\|\vert\boldsymbol{u}-\boldsymbol{u}_h \vert\|_h^2
		+
		\|\boldsymbol{p}-\boldsymbol{p}_h \|^2_{L^2(\mathcal{T}_h)}
		+
		\eta^2( \boldsymbol{u}_{h}, \boldsymbol{p}_{h};\mathcal{T}_{h})\\
		&&
		\quad  \lesssim  \|\vert\boldsymbol{u}-\boldsymbol{u}_h \vert\|_h^2 +\eta^2( \boldsymbol{u}_{h}, \boldsymbol{p}_{h};\mathcal{T}_{h}).
	\end{eqnarray*}
	Secondly, it is shown by the definition of $\|\cdot\|_{DG}$
	\begin{equation*}
	\|\vert\boldsymbol{u}-\boldsymbol{u}_h \vert\|_h^2 +\eta^2( \boldsymbol{u}_{h}, \boldsymbol{p}_{h};\mathcal{T}_{h}) \leq \|(\boldsymbol{u}, \boldsymbol{p})-(\boldsymbol{u}_h, \boldsymbol{p}_h)\|^2_{DG} +\eta^2( \boldsymbol{u}_{h}, \boldsymbol{p}_{h};\mathcal{T}_{h}).
	\end{equation*}
	Thus, we next only  need to consider  the convergence of  $ \|\vert\boldsymbol{u}-\boldsymbol{u}_h \vert\|_h^2 +\eta^2( \boldsymbol{u}_{h}, \boldsymbol{p}_{h};\mathcal{T}_{h})$.
\end{remark}

We first show that the error plus some quantity reduces with a fixed factor on two successive meshes.
\begin{lemma}\label{lem:tsm}
	Given $\boldsymbol{f}\in \boldsymbol{L}^2(\Omega)$ and two tetrahedral mesh $\mathcal{T}_h$ and $\mathcal{T}_{H}$, where $\mathcal{T}_{H}\leq\mathcal{T}_{h}$.
	Let $(\boldsymbol{u}, \boldsymbol{p})\in \boldsymbol{U}\times  \boldsymbol{Q}$ be the solution of \eqref{Equ:3.6}-\eqref{Equ:3.7}, and $(\boldsymbol{u}_h, \boldsymbol{p}_h) \in  \boldsymbol{U}_h\times \boldsymbol{Q}_h $, $(\boldsymbol{u}_{H}, \boldsymbol{p}_{H})\in \boldsymbol{U}_{H} \times  \boldsymbol{Q}_{H}$ be the solutions of \eqref{Equ:3.16}-\eqref{Equ:3.17}, respectively. Then there exit two constants $ \delta_1 , \delta_2 \in(0,1)$, such that
	\begin{eqnarray} \nonumber
	\|\vert\boldsymbol{u}-\boldsymbol{u}_h \vert\|_h^2
	&\leq&
	( 1+\delta_1) \|\vert  \boldsymbol{u}-\boldsymbol{u}_H \vert\|_H^2 -\frac{1-\delta_2}{2} \|\vert\boldsymbol{u}_h -\boldsymbol{u}_H \vert\|_h^2\\
	&&+
	\frac{C_3}{ \delta_1\delta_2 \alpha} \bigg(\eta^2( \boldsymbol{u}_{h}, \boldsymbol{p}_{h};\mathcal{T}_{h})+ \eta^2( \boldsymbol{u}_{H}, \boldsymbol{p}_{H};\mathcal{T}_{H})\bigg). \label{Equ:3.30.1}
	\end{eqnarray}
	where $C_3$ depending on the $C_\mathcal{L}$.
\end{lemma}

\begin{proof}
	Choosing that $\boldsymbol{q}=\nabla\times \boldsymbol{v}$, and subtracting \eqref{Equ:3.6} from \eqref{Equ:3.7}, we obtain
	\begin{equation} \label{Equ:4.7}
	(\kappa\boldsymbol{u} , \boldsymbol{v}) +( \mu\nabla\times \boldsymbol{u} , \nabla\times \boldsymbol{v} ) =(\boldsymbol{f}, \boldsymbol{v}).
	\end{equation}
	Subtracting \eqref{Equ:3.20} from \eqref{Equ:4.7} with  $\boldsymbol{v} =\boldsymbol{v}_h=\boldsymbol{u}_h^c -\boldsymbol{u}_H^c$, and using $ [[ \boldsymbol{u}_h^c -\boldsymbol{u}_H^c]] =0$, we have
	\begin{eqnarray*}
		(\kappa(\boldsymbol{u} -\boldsymbol{u}_h), \boldsymbol{u}_h^c -\boldsymbol{u}_H^c)_{0 ,\mathcal{T}_h}
		+
		(\mu\nabla\times(\boldsymbol{u} -\boldsymbol{u}_h), \nabla\times(\boldsymbol{u}_h^c -\boldsymbol{u}_H^c))_{0 ,\mathcal{T}_h}\\
		+
		<[[ \boldsymbol{u}_h]] , \{\{  \mu\nabla\times (\boldsymbol{u}_h^c -\boldsymbol{u}_H^c) \}\} >_{\mathcal{F}_h}=0,
	\end{eqnarray*}
	which leads to
	\begin{eqnarray}\label{Eqn:*1}\nonumber
	(\kappa(\boldsymbol{u} -\boldsymbol{u}_h), \boldsymbol{u}_h^c -\boldsymbol{u}_H^c)_{0 ,\mathcal{T}_h}
	+
	(\mu\nabla\times(\boldsymbol{u} -\boldsymbol{u}_h), \nabla\times(\boldsymbol{u}_h^c -\boldsymbol{u}_H^c))_{0 ,\mathcal{T}_h}\\
	=
	-<[[ \boldsymbol{u}_h]] , \{\{   \mu\boldsymbol{u}_h^c -\boldsymbol{u}_H^c \}\} >_{\mathcal{F}_h}.
	\end{eqnarray}
	Using \eqref{Eqn:L} and \eqref{Eqn:se}, we have
	\begin{eqnarray}\nonumber
	<[[ \boldsymbol{u}_h]] , \{\{  \nabla\times (\boldsymbol{u}_h^c -\boldsymbol{u}_H^c) \}\} >_{\mathcal{F}_h}
	&=& 	
	(\mathcal{L}_h(\boldsymbol{u}_h ), \nabla\times(\boldsymbol{u}_h^c -\boldsymbol{u}_H^c))_{0 ,\mathcal{T}_h}		
	\\ \label{Eqn:*2}
	& &~\hspace{-1.8cm} \leq
	C_{\mathcal{L}} \|h^{-1/2}[[ \boldsymbol{u}_h ]] \|_{0 ,\mathcal{T}_h} \|\nabla\times(\boldsymbol{u}_h^c -\boldsymbol{u}_H^c)\|_{0 ,\mathcal{T}_h}.
	\end{eqnarray}
	
	Let $\boldsymbol{u}_h = \boldsymbol{u}_h^c+\boldsymbol{u}_h^{\bot}$ and $\boldsymbol{u}_H = \boldsymbol{u}_H^c+\boldsymbol{u}_H^{\bot}$, we have
	\begin{equation}\label{Eqn:*3}
	\boldsymbol{u}_h+\boldsymbol{u}_H^c- \boldsymbol{u}_h^c= \boldsymbol{u}_H-\boldsymbol{u}_H^\bot + \boldsymbol{u}_h^\bot,
	\end{equation}
	where $\boldsymbol{u}^c_H\in\boldsymbol{U}_H^c$, $\boldsymbol{u}^c_h\in\boldsymbol{U}_h^c$, $\boldsymbol{u}^{\bot}_H\in\boldsymbol{U}_H^{\bot}$, $\boldsymbol{u}^{\bot}_h\in\boldsymbol{U}_h^{\bot}$.
	By \eqref{Eqn:*3}, \eqref{Eqn:*1}, \eqref{Eqn:*2} and Young's inequality, we get
	\begin{eqnarray*}
		&&\|\vert\boldsymbol{u}-\boldsymbol{u}_h \vert\|_h^2 \\
		&& =\|\kappa (\boldsymbol{u}-\boldsymbol{u}_h)\|^2_{L^2(\Omega)}
		+
		\|  \nabla\times\mu (\boldsymbol{u}-\boldsymbol{u}_h)\|^2_{L^2(\Omega)}\\
		&& \quad+
		\sum\limits_{f\in\mathcal{F}_h} \alpha h_f^{-1} <[[( \boldsymbol{u}-\boldsymbol{u}_h)]]   , [[ \boldsymbol{u}-\boldsymbol{u}_h]]  >_{\mathcal{F}_h}   \\
		&&
		=
		\|\vert \boldsymbol{u}-\boldsymbol{u}_h -\boldsymbol{u}_H^c+\boldsymbol{u}_h^c\vert\|_h^2
		-
		\|\vert\boldsymbol{u}_h^c -\boldsymbol{u}_H^c\vert\|_h^2
		-
		2(\kappa(\boldsymbol{u} -\boldsymbol{u}_h), \boldsymbol{u}_h^c -\boldsymbol{u}_H^c)_{0 ,\mathcal{T}_h}
		\\
		&& \quad
		-
		2(\mu\nabla\times(\boldsymbol{u} -\boldsymbol{u}_h), \nabla\times(\boldsymbol{u}_h^c -\boldsymbol{u}_H^c))_{0 ,\mathcal{T}_h}
		\\
		&& \quad
		-2 \sum\limits_{f\in\mathcal{F}_h} \alpha h_f^{-1}< [[ (\boldsymbol{u} -\boldsymbol{u}_h)]] , [[ \boldsymbol{u}_h^c -\boldsymbol{u}_H^c]] >   \\
		&& \lesssim
		\|\vert  \boldsymbol{u}-\boldsymbol{u}_H\vert\|_H^2 +2\|\vert \boldsymbol{u}-\boldsymbol{u}_H \vert\|_H \|\vert \boldsymbol{u}_h^\bot-\boldsymbol{u}_H^\bot \vert\|_h + \|\vert \boldsymbol{u}_h^\bot-\boldsymbol{u}_H^\bot \vert\|_h^2
		- \|\vert\boldsymbol{u}_h^c -\boldsymbol{u}_H^c\vert\|_h^2 \\
		&&  \quad +2 \|h^{-1/2}[[ \boldsymbol{u}_h ]] \|_{0 ,\mathcal{T}_h} \|\nabla\times(\boldsymbol{u}_h^c -\boldsymbol{u}_H^c)\|_{0 ,\mathcal{T}_h} \\
		&& \leq ( 1+\delta_1) \|\vert  \boldsymbol{u}-\boldsymbol{u}_H\vert\|_H^2 +(1+\frac{1}{\delta_1}) \|\vert \boldsymbol{u}_h^\bot-\boldsymbol{u}_H^\bot \vert\|_h^2 -(1-\hat{\delta}_2C_{\mathcal{L}}) \|\vert\boldsymbol{u}_h^c -\boldsymbol{u}_H^c\vert\|_h^2 \\
		&& \quad + \frac{C_{\mathcal{L}}}{\hat{\delta}_2}\|h^{-1/2}[[ \boldsymbol{u}_h ]] \|^2_{0 ,\mathcal{T}_h}\\
		&&=( 1+\delta_1) \|\vert  \boldsymbol{u}-\boldsymbol{u}_H\vert\|_H^2 +(1+\frac{1}{\delta_1}) \|\vert \boldsymbol{u}_h^\bot-\boldsymbol{u}_H^\bot \vert\|_h^2 -(1-{\delta}_2) \|\vert\boldsymbol{u}_h^c -\boldsymbol{u}_H^c\vert\|_h^2 \\
		&& \quad + \frac{C^2_{\mathcal{L}}}{ {\delta}_2}\|h^{-1/2}[[ \boldsymbol{u}_h ]] \|^2_{0 ,\mathcal{T}_h},
	\end{eqnarray*}
	where $\delta_2 = \hat{\delta}_2C_{\mathcal{L}}$.
	Using $\boldsymbol{u}_H^c =\boldsymbol{u}_H-\boldsymbol{u}_H^\bot$, $  \boldsymbol{u}_h^c =\boldsymbol{u}_h-\boldsymbol{u}_h^\bot$, triangle inequality and average inequality, we have
	\begin{eqnarray*}
		\|\vert\boldsymbol{u}_h^c -\boldsymbol{u}_H^c\vert\|_h^2 \geq    \frac{1}{2}\|\vert\boldsymbol{u}_h -\boldsymbol{u}_H\vert\|_h^2-   \|\vert\boldsymbol{u}_h^\bot -\boldsymbol{u}_H^\bot\vert\|_h^2.
	\end{eqnarray*}
	By triangle inequality and \eqref{Eqn:bot}, we obtain
	\begin{eqnarray*}
		\|\vert \boldsymbol{u}_h^\bot-\boldsymbol{u}_H^\bot \vert\|_h^2 & \leq & 2( \|\vert \boldsymbol{u}_h^\bot \vert\|_h^2+ \|\vert \boldsymbol{u}_H^\bot \vert\|_H^2)\\
		& \leq&  2\alpha\|h^{-1/2}[[ \boldsymbol{u}_h^\bot ]] \|^2_{0 ,\mathcal{T}_h} +2\alpha\|h^{-1/2}[[ \boldsymbol{u}_H^\bot ]] \|^2_{0 ,\mathcal{T}_h} .
	\end{eqnarray*}
	Combining  $[[ \boldsymbol{u}_H ]] =[[ \boldsymbol{u}_H^\bot+ \boldsymbol{u}_H^c ]] =[[ \boldsymbol{u}_H^\bot ]] $ and \eqref{Eqn:ajump}, we have
	\begin{eqnarray*}
		\|\vert\boldsymbol{u}-\boldsymbol{u}_h \vert\|^2_h
		&\leq& ( 1+\delta_1) \|\vert  \boldsymbol{u}-\boldsymbol{u}_H\vert\|_H^2 -\frac{1-\delta_2}{2}\|\vert\boldsymbol{u}_h -\boldsymbol{u}_H\vert\|_h^2 \\
		&&+  \frac{C_3}{ \delta_1\delta_2 \alpha} \bigg(\eta^2( \boldsymbol{u}_{h}, \boldsymbol{p}_{h};\mathcal{T}_{h})+ \eta^2( \boldsymbol{u}_{H}, \boldsymbol{p}_{H};\mathcal{T}_{H})\bigg).
	\end{eqnarray*}
\end{proof}

\subsection{Contraction of the error estimator}

In this subsection, we prove the reduction of error indicators. Let us first consider the effect of changing the finite element function used in the estimator.
\begin{lemma}\label{Lem:3.8}
	Given $\boldsymbol{f}\in \boldsymbol{L}^2(\Omega)$ and two tetrahedral mesh $\mathcal{T}_h$, $\mathcal{T}_{H} $ with $\mathcal{T}_{H}\leq\mathcal{T}_{h}$. Let $(\boldsymbol{v}_h, \boldsymbol{q}_h)\in \boldsymbol{U}_h \times \boldsymbol{Q}_h $ and $(\boldsymbol{v}_H, \boldsymbol{q}_H)\in \boldsymbol{U}_H \times \boldsymbol{Q}_H $. For any $\epsilon >0$, we have
	\begin{eqnarray}\label{Equ:4.2}
	\eta^2( \boldsymbol{v}_{h}, \boldsymbol{q}_{h};\mathcal{T}_{h})
	\leq (1+\epsilon) \eta^2( \boldsymbol{v}_{H}, \boldsymbol{q}_{H};\mathcal{T}_{h})+C_\epsilon \|(\boldsymbol{v}_{h}, \boldsymbol{q}_{h})-(\boldsymbol{v}_{H}, \boldsymbol{q}_{H})\|^2_{DG},
	\end{eqnarray}
	where $ C_\epsilon$ depending on  the $\epsilon$, and the mesh size $h<1$.
\end{lemma}

\begin{proof}	
	For any  $\tau_* \in \mathcal{T}_{h}$, we will discuss each of the five components of the mark $ \eta^2( \boldsymbol{v}_{h}, \boldsymbol{q}_{h};\mathcal{T}_{h})$.
	
	Firstly, using the definition of $R_1(\boldsymbol{v}_{h}, \boldsymbol{q}_{h})$ and triangle inequality, we have
	\begin{eqnarray}\label{Equ:3.41}
	\lefteqn{
		\|R_1(\boldsymbol{v}_{h}, \boldsymbol{q}_{h})\|_{L^2(\tau_*)}  }\\ \nonumber
	&&  =\|\boldsymbol{q}_{h}- \mu\nabla \times \boldsymbol{v}_{h}\|_{L^2(\tau_*)}\\\nonumber
	&&  =\|\boldsymbol{q}_{h}-\boldsymbol{q}_H + \mu\nabla \times (\boldsymbol{v}_{H} -  \boldsymbol{v}_{h})+\boldsymbol{q}_H- \mu\nabla \times \boldsymbol{v}_{H}\|_{L^2(\tau_*)}\\\nonumber
	&& \lesssim \|\boldsymbol{q}_H-\nabla \times \boldsymbol{v}_{H}\|_{L^2(\tau_*)}+\|\boldsymbol{q}_{h}-\boldsymbol{q}_H\|_{L^2(\tau_*)}+\|\nabla \times (\boldsymbol{v}_{h} -  \boldsymbol{v}_{H})\|_{L^2(\tau_*)}.
	\end{eqnarray}
	
	Secondly, using the definition of $R_2(\boldsymbol{v}_h,\boldsymbol{q}_h)$, triangle inequality and inverse inequality, we get
	\begin{eqnarray}\label{Equ:3.42}
	\lefteqn{
		h_{\tau_*}\|R_2(\boldsymbol{v}_{h}, \boldsymbol{q}_{h})\|_{L^2(\tau_*)}}\\ \nonumber
	&&=h_{\tau_*}(\|\boldsymbol{f}
	- \nabla\times \boldsymbol{q}_{h}
	- \kappa\boldsymbol{v}_{h}\|_{L^2(\tau_*)})\\\nonumber
	&&=h_{\tau_*}(\|\boldsymbol{f}
	-
	\nabla\times( \boldsymbol{q}_{h} -\boldsymbol{q}_H )
	- \kappa(\boldsymbol{v}_{h}-\boldsymbol{v}_{H})
	- \nabla\times \boldsymbol{q}_H
	- \kappa\boldsymbol{v}_{H}\|_{L^2(\tau_*)})\\\nonumber
	&&\leq h_{\tau_*}( \|\boldsymbol{f}-\nabla\times \boldsymbol{q}_{H}- \kappa\boldsymbol{v}_{H}\|_{L^2(\tau_*)}
	+
	\|\nabla\times( \boldsymbol{q}_h-\boldsymbol{q}_{H} )\|_{L^2(\tau_*)}
	+
	\|\kappa(\boldsymbol{v}_{h}-\boldsymbol{v}_{H})\|_{L^2(\tau_*)} )\\\nonumber
	&&\lesssim
	h_{\tau_*}(\|R_2(\boldsymbol{v}_{H}, \boldsymbol{q}_{H})\|_{L^2(\tau_*)}
	+
	h_{\tau_*}^{-1} \| (\boldsymbol{q}_h-\boldsymbol{q}_{H})\|_{L^2(\tau_*)}
	+
	\|\kappa(\boldsymbol{v}_{h}-\boldsymbol{v}_{H})\|_{L^2(\tau_*)} )\\\nonumber
	&&\lesssim
	h_{\tau_*}\|R_2(\boldsymbol{v}_{H}, \boldsymbol{q}_{H})\|_{L^2(\tau_*)}
	+
	\| (\boldsymbol{q}_h-\boldsymbol{q}_{H})\|_{L^2(\tau_*)}
	+
	h_{\tau_*}\|\kappa(\boldsymbol{v}_{h}-\boldsymbol{v}_{H})\|_{L^2(\tau_*)}.
	\end{eqnarray}
	
	Similarly, using the definition of $R_3(\boldsymbol{v}_h)$, triangle inequality and inverse inequality, we get
	\begin{eqnarray}\label{Equ:3.43}
	\lefteqn{
		h_{\tau_*}\|R_3(\boldsymbol{v}_{h})\|_{L^2(\tau_*)}  }\\ \nonumber
	&&=h_{\tau_*}\| \nabla\cdot (\boldsymbol{f}- \kappa \boldsymbol{v}_{h})\|_{L^2(\tau_*)}\\\nonumber
	&&=h_{\tau_*}\| \nabla\cdot (\boldsymbol{f}- \kappa\boldsymbol{v}_{H}+ \kappa\boldsymbol{v}_{H}- \kappa\boldsymbol{v}_{h})\|_{L^2(\tau_*)}\\\nonumber
	&&\leq h_{\tau_*}( \| \nabla\cdot (\boldsymbol{f}- \kappa\boldsymbol{v}_{H})\|_{L^2(\tau_*)}+\| \nabla\cdot\kappa ( \boldsymbol{v}_{H}- \boldsymbol{v}_{h})\|_{L^2(\tau_*)} )\\\nonumber
	&&\lesssim  h_{\tau_*}(\|R_3(\boldsymbol{v}_{H})\|_{L^2(\tau_*)}+h_{\tau_*}^{-1} \|\kappa(\boldsymbol{v}_{H}-\boldsymbol{v}_{h})\|_{L^2(\tau_*)})\\\nonumber
	&&\lesssim
	h_{\tau_*}\|R_3(\boldsymbol{v}_{H})\|_{L^2(\tau_*)}+ \|\kappa(\boldsymbol{v}_{H}-\boldsymbol{v}_{h})\|_{L^2(\tau_*)} .
	\end{eqnarray}
	
	Next, we discuss the jump $ J_1(\boldsymbol{q}_{h})$ and $J_2(\boldsymbol{v}_{h})$.
	For any $f\in \mathcal{F}(\mathcal{T}_h)$, we let $f= \tau_*^1 \bigcap  \tau_*^2$ with  $ \tau_*^1 , \tau_*^2 \in \mathcal{T}_h.$ Furthermore, using the definition of $J_1(\boldsymbol{q}_h)$, triangle inequality and trace inequality, we have
	\begin{eqnarray}\label{Equ:3.44}
	\lefteqn{
		h^{1/2}_f\|J_1(\boldsymbol{q}_{h})\|_{L^2(f)}}\\ \nonumber
	&&=h^{1/2}_f\|[[ \boldsymbol{q}_{h}]] \|_{L^2(f)}\\\nonumber
	&&=h^{1/2}_f\|[[ \boldsymbol{q}_{H}+\boldsymbol{q}_{h}-\boldsymbol{q}_H]] \|_{L^2(f)}\\\nonumber
	&&\leq h^{1/2}_f( \|[[ \boldsymbol{q}_{H}]] \|_{L^2(f)}+  \|[[ \boldsymbol{q}_{h}-\boldsymbol{q}_H]] \|_{L^2(f)} )\\\nonumber
	&&\leq h^{1/2}_f\|[[ \boldsymbol{q}_{H}]] \|_{L^2(f)}+h^{1/2}_f\|(\boldsymbol{q}_{h}-\boldsymbol{q}_H)\vert_{\tau_*^1}\|_{L^2(f)}+h^{1/2}_f\|(\boldsymbol{q}_{h}-\boldsymbol{q}_H)\vert_{\tau_*^2}\|_{L^2(f)}\\\nonumber
	&&\lesssim h^{1/2}_f\|J_1(\boldsymbol{q}_{H})\|_{L^2(f)}+ \|(\boldsymbol{q}_{h}-\boldsymbol{q}_H)\|_{L^2(\tau_*^1\cup \tau_*^2)}.
	\end{eqnarray}
	
	Similarly, using the definition of $J_2(\boldsymbol{v}_h)$, triangle inequality and trace inequality, we have
	\begin{eqnarray}\label{Equ:3.45}
	\lefteqn{
		h^{1/2}_f\|J_2(\boldsymbol{v}_{h})\|_{L^2(f)}  }\\ \nonumber
	&&=h^{1/2}_f\|[[ (\boldsymbol{f}-\kappa\boldsymbol{v}_{h})]] \|_{L^2(f)}\\\nonumber
	&& =h^{1/2}_f\|[[ (\boldsymbol{f}-\kappa\boldsymbol{v}_{H}+\kappa\boldsymbol{v}_{H}-\kappa\boldsymbol{v}_{h})]] \|_{L^2(f)}\\\nonumber
	&& \leq
	h^{1/2}_f(\|[[ (\boldsymbol{f}-\kappa\boldsymbol{v}_{H})]] \|_{L^2(f)}+\|[[ \kappa(\boldsymbol{v}_{H}-\boldsymbol{v}_{h})]] \|_{L^2(f)})\\\nonumber
	&& \leq h^{1/2}_f \|J_2(\boldsymbol{v}_{H})\|_{L^2(f)} +h^{1/2}_f(\|\kappa(\boldsymbol{v}_{H}-\boldsymbol{v}_{h})\vert_{\tau_*^1}\|_{L^2(f)} +\|\kappa(\boldsymbol{v}_{H}-\boldsymbol{v}_{h})\vert_{\tau_*^2}\|_{L^2(f)})\\\nonumber
	&& \lesssim
	h^{1/2}_f \|J_2(\boldsymbol{v}_{H})\|_{L^2(f)} +\|\kappa\boldsymbol{v}_{H}-\kappa\boldsymbol{v}_{h}\|_{L^2(\tau_*^1\cup \tau_*^2)}.
	\end{eqnarray}
	
	Finally, the desired result \eqref{Equ:4.2} is obtained by  combining \eqref{Equ:3.41}-\eqref{Equ:3.45}, Young's inequality and  the shape regularity of mesh $\mathcal{T}_h$.
\end{proof}

We then prove the contraction of the error estimator  under the assumptions on the problem of \eqref{Equ:3.16}-\eqref{Equ:3.17}.

\begin{lemma}\label{Lem:3.9}
	Given constant $\theta\in(0, 1)$  and two tetrahedral mesh $\mathcal{T}_h$, $\mathcal{T}_{H}(\mathcal{T}_{H}\leq\mathcal{T}_{h})$. Let $(\boldsymbol{u}_H, \boldsymbol{p}_H)\in \boldsymbol{U}_H\times  \boldsymbol{Q}_H$ be the solution of  \eqref{Equ:3.16}-\eqref{Equ:3.17},
	and $\mathcal{R}_{\mathcal{T}_H \longrightarrow \mathcal{T}_{h}}=\mathcal{T}_H\setminus(\mathcal{T}_{h}\cap\mathcal{T}_H)$ be the set of all element refined into $\mathcal{T}_{h}$ on $\mathcal{T}_H$. Then, there is a constant $ \lambda\in (0, 1)$ independent of mesh size, such that
	\begin{equation}
	\eta^2( \boldsymbol{u}_H, \boldsymbol{p}_H; \mathcal{T}_{h} ) \leq  \eta^2( \boldsymbol{u}_H, \boldsymbol{p}_H; \mathcal{T}_H )-  \lambda \eta^2( \boldsymbol{u}_H, \boldsymbol{p}_H;\mathcal{R}_{\mathcal{T}_H\rightarrow \mathcal{T}_{h}}).
	\end{equation}
	\end {lemma}
	
	\begin{proof}
		Assume that the tetrahedral mesh $\tau\in \mathcal{T}_H$  is divided into two new tetrahedral mesh $\tau_*^1$ and  $\tau_*^2$   with equal volumes, where $\tau_*^1, \tau_*^2 \in \mathcal{T}_{h}$. Thus, $h_{\tau_*^1}^3= \vert\tau_*^1\vert=  \vert\tau_*^2\vert = h_{\tau_*^2}^3 = 2^{-1} h_{\tau}^3$ by the shape regularity of mesh, which implies $h_{\tau_*^1} = h_{\tau_*^2}=2^{-1/3} h_{\tau}$. Then, we have
		\begin{equation}\label{Equ:3.48}
		\|R_1(\boldsymbol{u}_{H}, \boldsymbol{p}_{H})\|^2_{L^2(\tau_*^1)}
		+
		\|R_1(\boldsymbol{u}_{H}, \boldsymbol{p}_{H})\|^2_{L^2(\tau_*^2)}
		\leq \|R_1(\boldsymbol{u}_{H}, \boldsymbol{p}_{H})\|^2_{L^2(\tau)},
		\end{equation}
		and
		\begin{eqnarray}\nonumber
		&& h^2_{\tau_*^1}(\|R_2 (\boldsymbol{u}_H, \boldsymbol{p}_H)\|^2_{L^2(\tau_*^1)} +\|R_3(\boldsymbol{u}_H)\|^2_{L^2(\tau_*^1)} )\\
		&&~+h^2_{\tau_*^2}(\|R_2 (\boldsymbol{u}_H, \boldsymbol{p}_H)\|^2_{L^2(\tau_*^2)}  +\|R_3(\boldsymbol{u}_H)\|^2_{L^2(\tau_*^2)} )\nonumber \\\label{Equ:3.49}
		&& \quad \leq  2^{-2/3} h^2_{\tau}(\|R_2 (\boldsymbol{u}_H, \boldsymbol{p}_H)\|^2_{L^2(\tau)} +\|R_3(\boldsymbol{u}_H)\|^2_{L^2(\tau)} ) .
		\end{eqnarray}
		
		For any $f\in  \partial (\tau_*^1\cup\tau_*^2)$, which can be divided into three parts;
		
		(1) For the first part, there are two of the faces are constant and belong to $\tau$ .
		
		(2) For the second part, there are two new faces that overlap and are used to divide the mesh $\tau$. Since $(\boldsymbol{u}_H, \boldsymbol{p}_h)\in \boldsymbol{U}_H\times  \boldsymbol{Q}_H$ is a continuous polynomial in the region $\tau$, it follows that the value of $[[ \boldsymbol{p}_h]]$ and $[[ (\boldsymbol{f}-\kappa\boldsymbol{u}_H)]] $  on this surface is equal to zero.
		
		(3) For the third part, there are four faces that are obtained by dividing the two faces in the $\tau$ into two.
		
		Furthermore, we obtain
		\begin{equation}\label{Equ:3.51}
		\eta^2( \boldsymbol{u}_H, \boldsymbol{p}_H;\tau_*^1)+ \eta^2( \boldsymbol{u}_H, \boldsymbol{p}_H;\tau_*^2)\leq \overline{\gamma}  \eta^2( \boldsymbol{u}_H, \boldsymbol{p}_H;\tau).
		\end{equation}
		where constant $\overline{\gamma} \in (0, 1)$ independent of mesh $\tau$.
		
		Next, since $\mathcal{R}_{\mathcal{T}_H\rightarrow \mathcal{T}_{h}}$ represents the part of the set  in the tetrahedral set $\mathcal{T}_H $ that will be used to be refined, it follows that $\mathcal{R}_{\mathcal{T}_H\rightarrow \mathcal{T}_{h}} \subset  \mathcal{T}_H $. Let $ \overline{\mathcal{R}_{\mathcal{T}_H\rightarrow \mathcal{T}_{h}}}  $  denote the part of the cell set that has been refined in the tetrahedral set  $\mathcal{T}_{H} $, we have $\overline{\mathcal{R}_{\mathcal{T}_h\rightarrow \mathcal{T}_{H}}} \in  \mathcal{T}_{h}$. Obviously, $\mathcal{T}_H \setminus \mathcal{R}_{\mathcal{T}_H\rightarrow \mathcal{T}_{h}} =\mathcal{T}_{h} \setminus  \overline{\mathcal{R}_{\mathcal{T}_H\rightarrow \mathcal{T}_{h}}}$.
		Then combining the \eqref{Equ:3.51}, and the marking strategy \eqref{Equ:4.1}, we have
		\begin{eqnarray*}
			\eta^2( \boldsymbol{u}_H, \boldsymbol{p}_H; \mathcal{T}_{h} )
			&=& \eta^2( \boldsymbol{u}_H, \boldsymbol{p}_H; \mathcal{T}_{h} \setminus  \overline{\mathcal{R}_{\mathcal{T}_H\rightarrow \mathcal{T}_{h}}})+\eta^2( \boldsymbol{u}_H, \boldsymbol{p}_H; \overline{\mathcal{R}_{\mathcal{T}_H\rightarrow \mathcal{T}_{h}}})\\
			&\leq&
			\eta^2( \boldsymbol{u}_H, \boldsymbol{p}_H; \mathcal{T}_H \setminus \mathcal{R}_{\mathcal{T}_H\rightarrow \mathcal{T}_{h}} )+ \gamma  \eta^2( \boldsymbol{u}_H, \boldsymbol{p}_H;\mathcal{R}_{\mathcal{T}_H\rightarrow \mathcal{T}_{h}} )\\
			&\leq&
			\eta^2( \boldsymbol{u}_H, \boldsymbol{p}_H; \mathcal{T}_H )+(\overline{\gamma }-1) \eta^2( \boldsymbol{u}_H, \boldsymbol{p}_H;\mathcal{R}_{\mathcal{T}_H\rightarrow \mathcal{T}_{h}} )\\
			&\leq&
			\eta^2( \boldsymbol{u}_H, \boldsymbol{p}_H; \mathcal{T}_H )-  \lambda \eta^2( \boldsymbol{u}_H, \boldsymbol{p}_H;\mathcal{R}_{\mathcal{T}_H\rightarrow \mathcal{T}_{h}}),
		\end{eqnarray*}
		where $ \lambda =1-\overline{\gamma}  \in(0, 1)$ independent of mesh size.
	\end{proof}
	
	Now, we combine the Lemmas \ref{Lem:add1},  \ref{Lem:3.8} and \ref{Lem:3.9} to prove the reduction of error indicators.
	\begin{lemma}\label{Lem:3.11}
		Given a constant $\theta\in(0, 1)$ and  two tetrahedral mesh  $\mathcal{T}_h$, $\mathcal{T}_{H} (\mathcal{T}_{H}\leq\mathcal{T}_{h})$. Let $(\boldsymbol{u}_h, \boldsymbol{p}_h)\in \boldsymbol{U}_h\times  \boldsymbol{Q}_h$ and $(\boldsymbol{u}_{H}, \boldsymbol{p}_{H})\in \boldsymbol{U}_{H}\times  \boldsymbol{Q}_{H}$ be the solutions of  \eqref{Equ:3.16}-\eqref{Equ:3.17}, respectively. For any $\epsilon>0$ and $\lambda\in(0,1)$, we have
		\begin{eqnarray}\nonumber
		(1-\frac{C_\epsilon}{\alpha}  ) \eta^2( \boldsymbol{u}_{h}, \boldsymbol{p}_{h}; \mathcal{T}_{h} )
		&\leq&  ( 1+\epsilon+  \frac{C_\epsilon}{\alpha}  )\eta^2( \boldsymbol{u}_H, \boldsymbol{p}_H; \mathcal{T}_H )
		\\ \label{Equ:3.47}
		&&~\hspace{-1.8cm}
		-(1+\epsilon)\lambda \eta^2( \boldsymbol{u}_H, \boldsymbol{p}_H;\mathcal{R}_{\mathcal{T}_H\rightarrow \mathcal{T}_{h}})
		+
		C_\epsilon  \|\vert\boldsymbol{u}_{h}-\boldsymbol{u}_{H}\vert\|^2_h,
		\end{eqnarray}
		where constant $C_\epsilon$ depending on  the $\epsilon$ and mesh size.
	\end{lemma}
	
	\begin{proof}
		Using the Lemmas \ref{Lem:add1},  \ref{Lem:3.8} and \ref{Lem:3.9}, we have
		\begin{eqnarray*}
			\eta^2( \boldsymbol{u}_{h}, \boldsymbol{p}_{h}; \mathcal{T}_{h})
			&\leq&
			(1+\epsilon) \bigg(\eta^2( \boldsymbol{u}_H, \boldsymbol{p}_H; \mathcal{T}_H ) -\lambda \eta^2( \boldsymbol{u}_H, \boldsymbol{p}_H;\mathcal{R}_{\mathcal{T}_H\rightarrow \mathcal{T}_{h}})\bigg)
			\\
			&&+C_\epsilon \|(\boldsymbol{u}_{h}, \boldsymbol{p}_{h})-(\boldsymbol{u}_{H}, \boldsymbol{p}_{H})\|^2_{DG}\\
			&\leq&   (1+\epsilon)  \bigg(\eta^2( \boldsymbol{u}_H, \boldsymbol{p}_H; \mathcal{T}_H ) -\lambda \eta^2( \boldsymbol{u}_H, \boldsymbol{p}_H;\mathcal{R}_{\mathcal{T}_H\rightarrow \mathcal{T}_{h}})\bigg)
			\\
			&&+ C_\epsilon  \|\vert\boldsymbol{u}_{h}-\boldsymbol{u}_{H}\vert\|_h^2+ \|\boldsymbol{p}_h-\boldsymbol{p}_H\|^2_{L^2(\Omega)}\\
			&\leq&   (1+\epsilon)  \bigg(\eta^2( \boldsymbol{u}_H, \boldsymbol{p}_H; \mathcal{T}_H ) -\lambda \eta^2( \boldsymbol{u}_H, \boldsymbol{p}_H;\mathcal{R}_{\mathcal{T}_H\rightarrow \mathcal{T}_{h}})\bigg)\\
			&&+  C_\epsilon  \|\vert\boldsymbol{u}_{h}-\boldsymbol{u}_{H}\vert\|_h^2
			+ \frac{C_\epsilon}{\alpha} \bigg(\eta^2( \boldsymbol{u}_{h}, \boldsymbol{p}_{h}; \mathcal{T}_{h} ) +\eta^2( \boldsymbol{u}_H, \boldsymbol{p}_H; \mathcal{T}_H )\bigg),
		\end{eqnarray*}
		which completes the proof.	
	\end{proof}
	
	\subsection{Convergence result}
	
	Now, we proved that the sum of the norm of the error and the scaled error indicator is attenuated.		
	\begin{theorem}\label{Thm:3.3}
		For a given $\theta\in (0,1)$,let $\{  \mathcal{T}_k , \boldsymbol{U}_k,\boldsymbol{Q}_k, \boldsymbol{u}_{k}, \boldsymbol{p}_k, \eta( \boldsymbol{u}_{k}, \boldsymbol{p}_{k}; \mathcal{T}_{k} ) \}_{k\geq 0}$ be the sequence of meshes, Mixed discrete solution (defined by \eqref{Equ:3.16}-\eqref{Equ:3.17}), and the estimate indicator produced by the  \textbf{AMIPDG} algorithm. Then there exist constants $\rho>0$,  $\delta \in(0, 1)$, which depend on marking parameter $\theta$ and the shape regularity of the initial mesh $\mathcal{T}_0$, such that
		\begin{equation*}
		\|\vert\boldsymbol{u}-\boldsymbol{u}_{k+1} \vert\|_{k+1}^2  + \rho   \eta^2( \boldsymbol{u}_{k+1}, \boldsymbol{p}_{k+1}; \mathcal{T}_{k+1} ) \leq  \delta\bigg(\|\vert  \boldsymbol{u}-\boldsymbol{u}_k\vert\|_k^2+ \rho \eta^2( \boldsymbol{u}_k, \boldsymbol{p}_k; \mathcal{T}_k )\bigg).
		\end{equation*}
	\end{theorem}
	\begin{proof}
		Setting $ \widetilde{\rho}=\frac{1-\delta_2 }{2C_{\epsilon}} $, then multiply the both sides of the \eqref{Equ:3.47} inequality by $\widetilde{\rho}$, we get
		\begin{eqnarray} \nonumber
		\lefteqn{
			\widetilde{\rho}  (1-\frac{C_\epsilon}{\alpha}  )  \eta^2( \boldsymbol{u}_{k+1}, \boldsymbol{p}_{k+1}; \mathcal{T}_{k+1} ) }\\ \nonumber
		&&\leq   \widetilde{\rho}( 1+\epsilon+  \frac{C_\epsilon}{\alpha}  )  \eta^2( \boldsymbol{u}_k, \boldsymbol{p}_k; \mathcal{T}_k ) -\widetilde{\rho}(1+\epsilon)\lambda \eta^2( \boldsymbol{u}_k, \boldsymbol{p}_k;\mathcal{R}_{\mathcal{T}_k\rightarrow \mathcal{T}_{k+1}}) \\\label{Equ:4.15}
		&& \quad +\frac{1-\delta_2}{2}   \|\vert\boldsymbol{u}_{k+1}-\boldsymbol{u}_{k}\vert\|^2_h.
		\end{eqnarray}
		Next, by the  \eqref{Equ:3.30.1} and \eqref{Equ:4.15}, we have
		\begin{eqnarray}\nonumber
		\lefteqn{
			\|\vert\boldsymbol{u}-\boldsymbol{u}_{k+1} \vert\|_{k+1}^2  + \widetilde{\rho}  (1-\frac{C_\epsilon}{\alpha}  ) \eta^2( \boldsymbol{u}_{k+1}, \boldsymbol{p}_{k+1}; \mathcal{T}_{k+1} )}\\\nonumber
		&& \leq ( 1+\delta_1) \|\vert  \boldsymbol{u}-\boldsymbol{u}_k\vert\|_{k}^2 +  \frac{C_3}{ \delta_1\delta_2 \alpha} \bigg(\eta^2( \boldsymbol{v}_{k+1}, \boldsymbol{q}_{k+1};\mathcal{T}_{k+1})+ \eta^2( \boldsymbol{v}_{k}, \boldsymbol{q}_{k};\mathcal{T}_{k})\bigg)\\\label{Equ:3.50}
		&& \quad + \widetilde{\rho}( 1+\epsilon+  \frac{C_\epsilon}{\alpha}  )  \eta^2( \boldsymbol{u}_k, \boldsymbol{p}_k; \mathcal{T}_k ) -\widetilde{\rho}(1+\epsilon)\lambda \eta^2( \boldsymbol{u}_k, \boldsymbol{p}_k;\mathcal{R}_{\mathcal{T}_k\rightarrow \mathcal{T}_{k+1}}) .
		\end{eqnarray}
		First move the term and then according to  D\"orfler marking strategy \eqref{Equ:4.1}, the Theorem \ref{Thm:1} and $\|\vert\cdot\vert\|_h \leq \|\cdot\|_{DG}$, we know $  -\eta^2( \boldsymbol{v}_k, \boldsymbol{q}_k;\mathcal{R}_{\mathcal{T}_k\rightarrow \mathcal{T}_{k+1}}) \leq  -\theta \eta^2( \boldsymbol{v}_k, \boldsymbol{q}_k;\mathcal{T}_k) $, then
		\begin{eqnarray*}\nonumber
			\|\vert\boldsymbol{u}-\boldsymbol{u}_{k+1} \vert\|_{k+1}^2  &+& \widetilde{\rho}  (1-\frac{C_\epsilon}{\alpha} -  \frac{C_3}{  \widetilde{\rho}\delta_1\delta_2 \alpha}  ) \eta^2( \boldsymbol{u}_{k+1}, \boldsymbol{p}_{k+1}; \mathcal{T}_{k+1} ) \\\nonumber
			&\leq& ( 1+\delta_1) \|\vert  \boldsymbol{u}-\boldsymbol{u}_k\vert\|_k^2 - \frac{\widetilde{\rho}(1+\epsilon)\lambda \theta }{2} \eta^2( \boldsymbol{u}_k, \boldsymbol{p}_k; \mathcal{T}_k )\\
			&& \quad + \widetilde{\rho} \bigg( 1+\epsilon+  \frac{C_\epsilon}{\alpha} +\frac{C_3}{  \widetilde{\rho}\delta_1\delta_2 \alpha} - \frac{(1+\epsilon)\lambda \theta }{2} \bigg)  \eta^2( \boldsymbol{u}_k, \boldsymbol{p}_k; \mathcal{T}_k )  \\
			&\leq& ( 1+\delta_1- \frac{\widetilde{\rho}(1+\epsilon)\lambda \theta C_1^{-1}}{2} )\|\vert  \boldsymbol{u}-\boldsymbol{u}_k\vert\|_k^2   \\
			&& \quad + \widetilde{\rho} \bigg( 1+\epsilon+  \frac{C_\epsilon}{\alpha} +\frac{C_3}{  \widetilde{\rho}\delta_1\delta_2 \alpha} - \frac{(1+\epsilon)\lambda \theta }{2} \bigg)  \eta^2( \boldsymbol{u}_k, \boldsymbol{p}_k; \mathcal{T}_k ).
		\end{eqnarray*}
		
		For convenience, denote
		\begin{eqnarray*}
			\beta_1 &=&1-\frac{C_\epsilon}{\alpha} -  \frac{C_3}{  \widetilde{\rho}\delta_1\delta_2 \alpha} , \\
			\beta_2 &=&   1+\delta_1- \frac{\widetilde{\rho}(1+\epsilon)\lambda \theta C_1^{-1} }{2} , \\
			\beta_3&=& (1+\epsilon)(1 -\frac{\lambda \theta }{2})  +\frac{C_\epsilon}{\alpha} +\frac{C_3}{  \widetilde{\rho}\delta_1\delta_2 \alpha} .
		\end{eqnarray*}
		Thus
		\begin{equation*}
		\|\vert\boldsymbol{u}-\boldsymbol{u}_{k+1} \vert\|_{k+1}^2  + \widetilde{\rho}  \beta_1  \eta^2( \boldsymbol{u}_{k+1}, \boldsymbol{p}_{k+1}; \mathcal{T}_{k+1} ) \leq  \beta_2 \|\vert  \boldsymbol{u}-\boldsymbol{u}_k\vert\|_k^2+ \widetilde{\rho}\beta_3 \eta^2( \boldsymbol{u}_k, \boldsymbol{p}_k; \mathcal{T}_k ).
		\end{equation*}
		
		Next, we firstly choose $ \delta_1=\frac{\widetilde{\rho}(1+\epsilon)\lambda \theta C_1^{-1} }{4}$, then select the appropriate $\delta_2$ to make $\widetilde{\rho}=\frac{1-\delta_2 }{2C_{\epsilon}} $ smaller to ensure $0< \delta_1<1$,
		Secondly, we let  $\epsilon>0$ and $ (1+\epsilon)(1 -\frac{\lambda \theta }{2})=1-\frac{\lambda \theta }{4}$ ( $\lambda \theta\in(0,1) $), therefore
		\begin{eqnarray*}
			\beta_2=1- \delta_1 \in(0,1),
			\
			(1+\epsilon)(1 -\frac{\lambda \theta }{2})<1.
		\end{eqnarray*}
		
		Furthermore, we choose a sufficiently large penalty parameter $\alpha$ such that
		$$\beta_1 >\beta_3.$$
		
		Finally, there is a constant $\delta=\max\{\beta_2,\frac{\beta_1}{\beta_3} \} $. Then, we let $\rho=  \widetilde{\rho}  \beta_1$, and obtain
		\begin{equation*}
		\|\vert\boldsymbol{u}-\boldsymbol{u}_{k+1} \vert\|_{k+1}^2  + \rho   \eta^2( \boldsymbol{u}_{k+1}, \boldsymbol{p}_{k+1}; \mathcal{T}_{k+1} ) \leq  \delta\bigg(\|\vert  \boldsymbol{u}-\boldsymbol{u}_k\vert\|_k^2+ \rho \eta^2( \boldsymbol{u}_k, \boldsymbol{p}_k; \mathcal{T}_k )\bigg).
		\end{equation*}			
	\end{proof}
	
	\begin{corollary}
		Under the conditions of Theorem \ref{Thm:3.3}, we have
		\begin{eqnarray*}
			\|(\boldsymbol{u}, \boldsymbol{p})-(\boldsymbol{u}_{k}, \boldsymbol{p}_{k})\|^2_{DG}  + \rho   \eta^2( \boldsymbol{u}_{k}, \boldsymbol{p}_{k}; \mathcal{T}_{k} ) \leq  \delta^k\widetilde{C}_\delta.
		\end{eqnarray*}
		where $ \widetilde{C}_\delta = C \bigg( \|(\boldsymbol{u}, \boldsymbol{p})-(\boldsymbol{u}_{0}, \boldsymbol{p}_{0})\|^2_{DG}  + \rho   \eta^2( \boldsymbol{u}_{0}, \boldsymbol{p}_{0}; \mathcal{T}_{0} ) \bigg) $.
		Therefore, for a given precision, the \textbf{AMIPDG} method will terminate after a finite number of operations.
	\end{corollary}
	
	\begin{proof}
		Using the Remark \ref{Rem:4} and Theorem \ref{Thm:3.3}, we have
		\begin{eqnarray*}
			\|(\boldsymbol{u}, \boldsymbol{p})-(\boldsymbol{u}_{k}, \boldsymbol{p}_{k})\|^2_{DG} + \rho   \eta^2( \boldsymbol{u}_{k}, \boldsymbol{p}_{k}; \mathcal{T}_{k} )
			&\leq&
			C \bigg(\|\vert  \boldsymbol{u}-\boldsymbol{u}_k\vert\|_k^2+ \rho \eta^2( \boldsymbol{u}_k, \boldsymbol{p}_k; \mathcal{T}_k )\bigg)\\
			&\leq&
			\delta^k\widetilde{C}_\delta.
		\end{eqnarray*}	
	\end{proof}

	\section{Numerical experiments}
	In this section, we test some numerical experiments to show the efficiency and the robustness of AMIPDG.
	We carry out these numerical experiments by using the MATLAB software package iFEM \cite{ChenLiFEM}. In Experiments \ref{Exa:ex1} and \ref{Exa:ex2}, we take $\boldsymbol{p}=\nabla\times\boldsymbol{u}$.
	
	In Example \ref{Exa:ex1}, we discuss the influence of the penalty parameter $\alpha$ on the error in  $\|\cdot\|_{DG}$ norm, and observe the dependency of the condition number of stiffness matrix on $\alpha$.
	\begin{example}\label{Exa:ex1}
		Let $\Omega:=[0,1] \times[0,1] \times [0,1]$, we construct the following analytical solution of the model \eqref{Equ:1.1}-\eqref{Equ:1.2}:
		$$
		\boldsymbol{u}=\left(\begin{array}{c}
		x(x-1)y(y-1)z(z-1) \\
		\sin (\pi x) \sin (\pi y) \sin (\pi z) \\
		(1-e^x)(1-e^{x-1})(1-e^{y})(1-e^{y-1})(1-e^{z})(1-e^{z-1})	
		\end{array}\right).
		$$
		It is easy to see that the solution $\boldsymbol{u}$ satisfies the boundary condition $\boldsymbol{u} \times \boldsymbol{n}=0$ on $\partial \Omega$.
	\end{example}
	
	In this example, we get a uniform mesh by partitioning the $x-$, $y-$ and $z-$axes into equally distributed $M(M\geq 2)$ subintervals,  and then dividing one cube into six tetrahedrons. Let $h=1/M$ be mesh sizes for different tetrahedrons meshes.
	We fixed mesh with $h=1/4$ and report the error estimates in  $\|\cdot\|_{DG}$  norm and condition number of stiffness matrices for different penalty parameters $ \alpha = 1, 10, 100, 500$ and $1000$ in Table \ref{Tab_111}. We note that $\left\|\boldsymbol{u}-\boldsymbol{u}_{h}\right\|_{0}$ increases at first and then decreases as the penalty parameter $\alpha$ increases. The condition numbers of stiffness matrices increase with the increase of penalty parameters $\alpha$.
	
	\begin{table}[ht]
		\centering\caption{The error in  $\|\cdot\|_{DG}$ norms and condition number of stiffness matrices with $h=1/4$.}\label{Tab_111}
		\begin{tabular}{ccccccc}
			\hline
			$\alpha$	& 1  & 10    & 100 & 500 & 1000 \\
			\hline
			$\|\left(\boldsymbol{p}-\boldsymbol{p}_{h}, \boldsymbol{u}-\boldsymbol{u}_{h}\right)\|_{\mathrm{DG}}$ & 3.949e+00 & 1.133e-00 &  8.614e-01 & 8.649e-01  & 8.659e-01  \\
			Cond &3.235e+04 & 7.021e+04 & 5.959e+05 & 2.995e+06 & 6.150e+06\\
			\hline
		\end{tabular}
	\end{table}
	
	As a way to balance, in the following numerical tests, we always choose $\alpha=100$.
	
	Noting that we only consider uniform meshes in Example \ref{Exa:ex1}. Next we test adaptive meshes.
	
	\begin{example}\label{Exa:ex2}
		Let $\Omega:=[0,1] \times[0,1] \times [0, 1]$, we construct the following analytical solution of the model \eqref{Equ:1.1}-\eqref{Equ:1.2}
		$$
		\boldsymbol{u}=\left(\begin{array}{c}
		\frac{x(x-1)y(y-1)z(z-1)}{x^2+y^2+z^2+0.001} \\
		\frac{x(x-1)y(y-1)z(z-1)}{x^2+y^2+z^2+0.001} \\
		-\frac{x(x-1)y(y-1)z(z-1)}{x^2+y^2+z^2+0.001}
		\end{array}\right).
		$$
		Note that the solution $\boldsymbol{u}$ satisfies the condition $ \boldsymbol{u}\times\boldsymbol{n}=0$ on $\partial \Omega$.
	\end{example}
	
	The right of Figure \ref{fig:ex2_mesh_figure} shows an adaptively refined mesh with marking parameter- $\theta=0.7$ after $k=18$. The grid is locally refined near the origin.
	\begin{figure}[ht]
		\begin{minipage}[t]{0.45\linewidth}
			\centering
			\includegraphics[height=4cm,width=5cm]{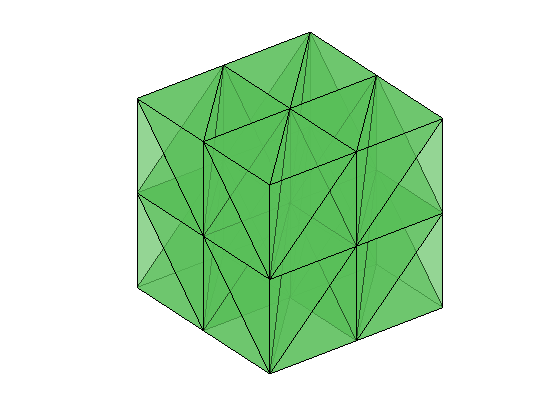}
		\end{minipage}
		\begin{minipage}[t]{0.45\linewidth}
			\centering
			\includegraphics[height=4cm,width=5cm]{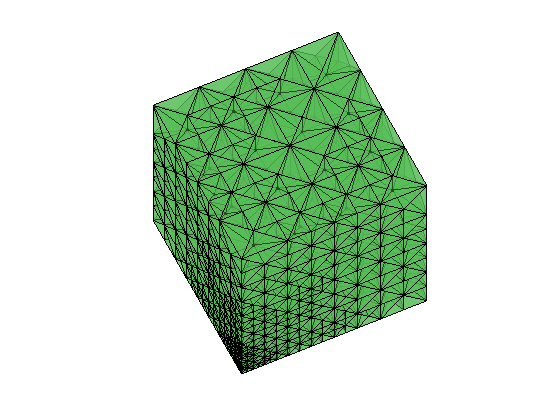}
		\end{minipage}
		\caption{Left: the initial mesh with 1152 DoFs. Right: the adaptive mesh($\theta=0.7$) with 181104 DoFs after 18 refinements.}\label{fig:ex2_mesh_figure}
	\end{figure}
	
	The Figure \ref{fig:ex2_res} shows the curves of $\log N -\log\eta\left(\boldsymbol{u}_{k}, \boldsymbol{p}_{k} ; \mathcal{T}_{k}\right)$ for parameters $\theta=0.3,0.5,0.7$. The curves indicate the convergence and the quasi-optimality of the adaptive algorithm AMIPDG of  $\eta\left(\boldsymbol{u}_{k}, \boldsymbol{p}_{k} ; \mathcal{T}_{k}\right)$.
	\begin{figure}[ht]
		\centering
		\includegraphics[height=8cm,width=12cm]{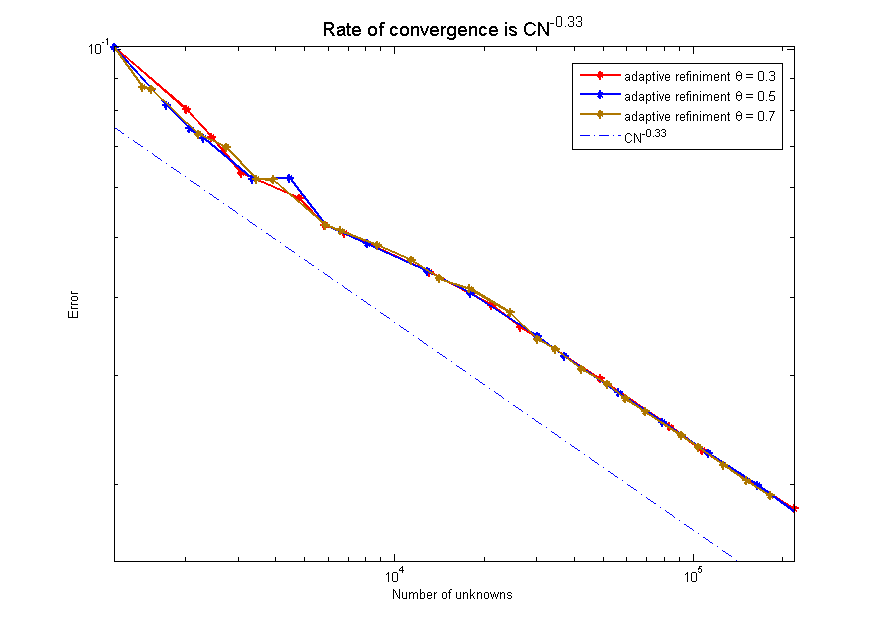}
		\caption{Quasi optimality of the AMIPDG of the error $\eta\left(\boldsymbol{u}_{k}, \boldsymbol{p}_{k} ; \mathcal{T}_{k}\right)$ with different marking parameters $\theta$.}\label{fig:ex2_res}
	\end{figure}
	
	\section*{Acknowledgment}
	The first author is supported by the East China University of Technology (DHBK2019209) and Jiangxi Province Education Department (GJJ200755).
	The second,  third and fourth authors are supported by the National Natural Science Foundation of China (Grant No. 12071160).   The third author is  also supported by the National Natural Science Foundation of China (Grant No. 11901212).


\end{document}